\documentclass[a4paper,11pt, twoside]{article}
\usepackage[utf8]{inputenc}
\usepackage{hyperref}
\usepackage{amsmath,amsthm, amssymb}
\usepackage{mathtools}
\usepackage{bbm}
\usepackage{graphics}
\usepackage{subcaption}
\usepackage[table]{xcolor}
\usepackage{float}
\usepackage{geometry}
\usepackage{comment}
\usepackage{mathrsfs}
\usepackage{tkz-euclide}
\usetikzlibrary{decorations.markings}
\usepackage{dsfont}
\usepackage{mdframed}
\usepackage{listings}
\usepackage{float}
\usepackage{enumitem}
\usepackage{empheq}
\usepackage{amssymb}
\usepackage{pifont}
\usepackage{calc}
\usepackage{tabularx}

\usepackage{siunitx}
\newcommand{\cmark}{\ding{51}}%
\newcommand{\xmark}{\ding{55}}%

\usepackage{tabularx}
\makeatletter
\def\hlinewd#1{%
\noalign{\ifnum0=`}\fi\hrule \@height #1 %
\futurelet\reserved@a\@xhline}
\makeatother

\makeatletter
\AtBeginDocument{%
  \expandafter\renewcommand\expandafter\paragraph\expandafter{%
    \expandafter\@fb@secFB\paragraph
  }%
}
\makeatother

\newcounter{counter}
\numberwithin{counter}{section}
\newtheorem{theorem}[counter]{Theorem}
\newtheorem{assumption}[counter]{Assumption}
\newtheorem*{remark}{Remark}

\newtheorem{definition}[counter]{Definition}

\newtheorem{lemma}[counter]{Lemma}

\numberwithin{equation}{section}

\makeatletter
\let\orgdescriptionlabel\descriptionlabel
\renewcommand*{\descriptionlabel}[1]{%
  \let\orglabel\label
  \let\label\@gobble
  \phantomsection
  \edef\@currentlabel{#1}%
  \let\label\orglabel
  \orgdescriptionlabel{#1}%
}
\makeatother

\newcommand{\abs}[1]{\left\lvert #1\right\rvert}

\newcommand{\OA}{\operatorname{\mathsf{A}}}
\newcommand{\OV}{\operatorname{\mathsf{V}}}

\newcommand{\OC}{\operatorname{\mathsf{C}}}
\newcommand{\OK}{\operatorname{\mathsf{K}}}
\newcommand{\OW}{\operatorname{\mathsf{W}}}
\newcommand{\OH}{\operatorname{\mathsf{H}}}
\newcommand{\Id}{\operatorname{\mathsf{Id}}}

\newcommand{\OR}{\operatorname{\mathsf{R}}}
\renewcommand{\OE}{\operatorname{\mathsf{E}}}
\renewcommand{\OH}{\operatorname{\mathsf{H}}}

\newcommand{\OI}{\mathcal{I}}

\newcommand{\MV}{\mathbf{V}_h}
\newcommand{\MW}{\mathbf{W}_h}
\newcommand{\MM}{\mathbf{M}_h}
\newcommand{\ME}{\mathbf{E}_h}
\newcommand{\MK}{\mathbf{K}_h}
\newcommand{\MH}{\mathbf{H}_h}
\newcommand{\MEi}{\mathbf{E}_{i,h}}

\newcommand{\calS}{\mathcal{S}}

\newcommand{\calL}{\mathcal{L}}

\renewcommand{\O}{\mathcal{O}}

\newcommand{\BH}{\mathbf{H}}

\newcommand{\Vx}{\mathbf{x}}

\newcommand{\Vn}{\mathbf{n}}

\newcommand{\Vv}{\mathbf{v}}

\newcommand{\Vu}{\mathbf{u}}

\newcommand{\Vgamma}{\boldsymbol{\gamma}}

\newcommand{\Vrho}{\boldsymbol{\rho}}

\newcommand{\R}{\mathbb{R}}

\renewcommand{\H}{\mathbb{H}}

\newcommand{\dual}[2]{\left\langle #1\,,\,#2\right\rangle}

\newcommand{\norm}[2]{\left\lVert #1\right\rVert_{#2}}
\definecolor{orange}{rgb}{1,0.4,0}
\definecolor{green}{rgb}{0,0.4,0}

\newcommand{\red}[1]{{\color{red}#1 }}

\newcommand{\orange}[1]{{\color{orange}#1 }}
\newcommand{\green}[1]{{\color{green}#1 }}

\def\d{\delta}

\newcommand{\curl}{\operatorname{\boldsymbol{\mathrm{curl}}}}

\renewcommand{\d}{\ensuremath{\mathrm{d}}}
\newcommand{\RR}{\mathbb{R}}
\newcommand{\NN}{\mathbb{N}}

\newcommand{\KK}{\mathbb{K}}
\newcommand{\GGG}{\mathcal{G}}

\newcommand{\SLk}{\Psi_{SL, k}}
\newcommand{\DLk}{\Psi_{DL, k}}

\newcommand{\vect}[1]{\boldsymbol{#1}}

\newcommand{\bcurl}{\textbf{curl}}

\renewcommand{\div}{\text{div}}
\newcommand{\Hdiv}{\vect{H}^{-1/2}_{\times}(\div_{\Gamma},\Gamma)}
\newcommand{\Hcurl}{\vect{H}^{-1/2}_{\times}(\bcurl_{\Gamma},\Gamma)}
\makeatletter
\newsavebox{\@brx}
\newcommand{\llangle}[1][]{\savebox{\@brx}{\(\m@th{#1\langle}\)}%
  \mathopen{\copy\@brx\kern-0.5\wd\@brx\usebox{\@brx}}}
\newcommand{\rrangle}[1][]{\savebox{\@brx}{\(\m@th{#1\rangle}\)}%
  \mathclose{\copy\@brx\kern-0.5\wd\@brx\usebox{\@brx}}}
\makeatother
\newcommand{\Ph}{\mathcal{P}_h}

\title{Convergence of Calder\'on residuals}
\author{R. Hiptmair$^1$, C. Urz\'ua-Torres$^2$, A. Wisse$^2$}
\date{
{\small 
$^1$Seminar for Applied Mathematics, ETH Zurich,\\
R\"amistrasse 101, 8092 Z\"urich, Switzerland.\\[3mm]
$^2$Delft Institute of Applied Mathematics, TU Delft, \\
Mekelweg 4, 2628CD Delft, The Netherlands }}

\begin{document}
\maketitle

\begin{abstract}
In this paper, we describe a framework to compute expected convergence rates for
residuals based on the Calder\'on identities for general second order differential 
operators for which fundamental solutions are known. The idea is that these rates 
could be used to validate implementations of boundary integral operators and 
allow to test operators separately by choosing solutions where parts of the 
Calder\'on identities vanish. Our estimates rely on simple vector norms, and 
thus avoid the use of hard-to-compute norms and the residual computation can be 
easily implemented in existing boundary element codes.
We test the proposed Calder\'on residuals as debugging tool by introducing 
artificial errors into the Galerkin matrices of some of the boundary integral 
operators for the Laplacian and time-harmonic Maxwell's equations. From this, we 
learn that our estimates are not sharp enough to always detect errors, but 
still provide a simple and useful debugging tool in many situations.
\end{abstract}


\section{Introduction}
Boundary integral equations (BIEs) and boundary element methods are popular when 
dealing with problems on unbounded domains and for which a fundamental solution 
is available. What is considerably less popular among practitioners and students 
is to validate and debug a boundary element code. On the one hand, the reality 
that one needs to handle singular integral adds complexity to the 
implementation (and thus additional places where things can go wrong). On the 
other hand, the fact that convergence rates for Galerkin discretizations of 
boundary integral equations are in fractional and even negative norms, make them 
an ``all-or-nothing-tool'' for code validation, and hence not truly useful for 
debugging. Indeed, if the Galerkin matrices are correctly implemented, one can 
typically choose a reference solution and use some of the matrices to measure 
the error via exploiting the norm equivalence between the solution space of the 
boundary integral equation and the energy norm of the operator. If the observed 
convergence rate agrees with the expected one, one knows that the implementation 
is correct. This, however, does not work if the implementation of the Galerkin 
matrix is incorrect, but how can we assess this in the first place?

In this paper, we describe a framework to compute expected convergence rates for 
residuals based on the Calder\'on identities for general second order differential 
operators for which fundamental solutions are known. The idea is that these rates 
could be used to 
validate implementations of boundary integral operators and allow to test operators 
separately by choosing solutions where parts of the Calder\'on identities vanish. 
Our estimates rely on simple vector norms, and thus avoid the use 
of hard-to-compute norms and the residual computation can be easily implemented in existing 
boundary element codes.\\
{\bf{Outline}}\\ To illustrate the idea, we consider the general case 
first and, after this, fill in the details for the Laplace equation and
time harmonic Maxwell's equations. We also provide several numerical experiments 
with correct and defective matrices to investigate the performance of the proposed
\emph{Calder\'on residuals} to detect implementation errors. It turns out that 
our estimates are not sharp enough to capture everything, but, given their
implementation simplicity, we believe they are still worth trying in some cases.

\section{General Case}
Let  $\Omega\subset\R^d$ for $d=2,3$ be a bounded domain with Lipschitz boundary 
$\Gamma = \partial \Omega$. We consider a second order partial differential
equation of the type
\begin{align}\begin{cases}
\label{eq:BVP}  
  &\calL u = 0 \quad\text{in } \R^d \setminus \overline{\Omega},\\
  &\text{+ boundary conditions on } \Gamma,\\
  &\text{+ radiation conditions at $\infty$,}
\end{cases}\end{align}
which we want to solve using boundary integral equations. 

Let $\OC$ be the Calderón projector associated to the BIEs for $\calL$ and 
$\Vgamma$ the Cauchy trace operator comprising of the Dirichlet-type trace 
$\gamma_0$ and the Neumann-type trace $\gamma_1$. Then, 
the corresponding exterior Calderón identity is given by 
\begin{align}\label{eq:general_Calderon_identity}
  \OC \vect{\gamma} = \vect{\gamma} \iff (\OC - \Id)\vect{\gamma} = 0,
\end{align}
where $\Id$ denotes the identity operator.
For now, we want to express this in the most general way possible. Therefore, 
we write
\begin{align*}
  \OC := \begin{pmatrix}
        \OA_1 + \tfrac{1}{2}\Id & \OA_2 \\
         \OA_3 & \OA_4 + \tfrac{1}{2}\Id
       \end{pmatrix},\quad \vect{\gamma} =\begin{pmatrix}
                                           \gamma_0 u \\\gamma_1 u 
                                         \end{pmatrix},
\end{align*}
where $\OA_i$ are the boundary integral operators (BIOs) related to 
$\mathcal{L}$ and $\Vgamma$. We have that
\begin{align*}
  \OA_1: X_0 \to X_0,\quad & \OA_2: X_1 \to X_0,\\
  \OA_3: X_0 \to X_1, \quad & \OA_4: X_1 \to X_1,
\end{align*}
where $X_0$ and $X_1$ are suitable Hilbert spaces on the boundary $\Gamma$ such 
that $X_0'$ and $X_1$ are isomorphic, and $\OA_1,\dots,\OA_4$ are continuous 
operators and the traces $\gamma_0$ and $\gamma_1$ are also continuous.

If $u$ is a solution to \eqref{eq:BVP}, then \eqref{eq:general_Calderon_identity} 
implies 
\begin{subequations}\label{eq:general_BIE}
\begin{align}
(\OA_1 - \tfrac{1}{2}\Id)\gamma_0 u + \OA_2 \gamma_1 u &=0 \quad \text{ in } X_0 \label{eq:general_BIE1},\\
\OA_3 \gamma_0 u+ (\OA_4 -\tfrac{1}{2}\Id)\gamma_1 u &=0 \quad \text{ in } X_1\label{eq:general_BIE2}
\end{align}
\end{subequations}
hold exactly. We can use \eqref{eq:general_BIE} for code validation by testing 
\begin{subequations}\label{eq:general_BIE_discrete}
\begin{align}
(\OA_1 - \tfrac{1}{2}\Id)\OI_0\gamma_0 u + \OA_2\OI_1 \gamma_1 u &\approx0 \label{eq:general_BIE1_discrete},\\
\OA_3 \OI_0\gamma_0 u+ (\OA_4 -\tfrac{1}{2}\Id)\OI_1\gamma_1 u &\approx0\label{eq:general_BIE2_discrete},
\end{align}
\end{subequations}
where $\OI_0$ and $\OI_1$ are operators that map the traces $\gamma_0 u$ and 
$\gamma_1 u$ to suitable finite dimensional subspaces. They can be interpolators 
or projections, but for simplicity we will refer to them as interpolants. Why 
``$\approx$" now? This is because in general the interpolants will not lie in 
the kernel of $(\OC - \Id)$. Yet we can expect them to be close, thus checking 
this can provide evidence for the correctness of a BEM code. More precisely, 
failing to see ``$\approx 0$" in \eqref{eq:general_BIE_discrete} will indicate 
errors.  

To grasp what this means in quantitative terms, let us consider a family of 
dyadically refined meshes $(\GGG_h)_{h\in \H}$ of $\Gamma$, where the index $h$ 
denotes the corresponding meshwidth. Then, for each $h\in \H$, let
\begin{equation*}
  X_{0,h}\subset X_0, \quad X_{1,h}\subset X_1
\end{equation*}
define (conforming) boundary element spaces on the mesh $\GGG_h$ such that 
the approximation of $X_i$ by $X_{i,h}$ improves when $h\to0$.
Next, we introduce interpolation operators (`boundary element interpolants') 
\begin{align*}
  \OI_0 : \mathcal{D}(X_0)\subset X_0 \to X_{0,h},\quad \OI_1: \mathcal{D}(X_1)
  \subset X_1 \to X_{1,h}
\end{align*}
mapping to the finite-dimensional subspaces. With this, we define the 
\emph{Calder\'on residuals} as 
\begin{subequations}
\begin{align}
r_0(\psi) &:= \int_{\Gamma} [(\OA_1 -\tfrac{1}{2}\Id)\OI_0 \gamma_0 u + \OA_2 
\OI_1 \gamma_1 u]\psi \;\d \Gamma, \quad \psi\in X_1,\label{eq:residuals1}\\
r_1(\phi) &:= \int_{\Gamma}[\OA_3 \OI_0 \gamma_0 u + (\OA_4 -\tfrac{1}{2}\Id)
\OI_1\gamma_1 u]\phi \;\d\Gamma,\quad \phi\in X_0.\label{eq:residuals2}
\end{align}
\end{subequations}
By the approximation property of our finite-dimensional subspaces, we expect these 
residuals to become small when $h \to 0$. In order to quantify how this happens, 
it is easier to consider the rewritten forms
\begin{align*}
  r_0(\psi) &= \int_{\Gamma} [(\OA_1 -\tfrac{1}{2}\Id)(\OI_0-\Id) \gamma_0 u + 
  \OA_2 (\OI_1-\Id) \gamma_1 u]\psi \;\d \Gamma ,\\
  r_1(\phi) &= \int_{\Gamma}[\OA_3 (\OI_0-\Id) \gamma_0 u + (\OA_4 -\tfrac{1}{2}
  \Id)(\OI_1-\Id)\gamma_1 u]\phi \;\d\Gamma,
\end{align*}
which are obtained by subtracting the weak forms of \eqref{eq:general_BIE1} and 
\eqref{eq:general_BIE2} from \eqref{eq:residuals1} and \eqref{eq:residuals2}, 
respectively, so that we can use existing results to compute a convergence rate.
In other words, we bound the residuals by
\begin{align}
\label{eq:EstimateIdea}
\abs{r_0(\psi)} \leq \left[ \right.&\norm{\OA_1 -\tfrac{1}{2}\Id}{X_0\to X_0}
\norm{(\OI_0-\Id) \gamma_0 u}{X_0}\nonumber\\ &+ \left.
\norm{\OA_2}{X_1\to X_0}\norm{(\OI_1-\Id) \gamma_1 u}{X_1}\right]\norm{\psi}
{X_0^\prime} , \nonumber\\
\abs{r_1(\phi)}\leq \left[ \right.& \norm{\OA_3}{X_0\to X_1} \norm{(\OI_0-\Id)
\gamma_0 u}{X_0} \nonumber\\ &+ \left.  
\norm{\OA_4 -\tfrac{1}{2} \Id}{X_1\to X_1}\norm{(\OI_1-\Id)\gamma_1 u}{X_1}
\right]\norm{\phi}{X_1^\prime} .
\end{align}

Since estimates of the type 
$$\norm{(\OI_i -\Id)w}{X_i} \leq h^{\alpha_i}\norm{w}{\mathcal{X}_i}, 
\quad i=0,1,$$
with $\norm{\cdot}{\mathcal{X}_i}$ some appropriate norm, are usually 
available, then it only remains to estimate the test functions when these are 
taken in the finite-dimensional subspaces as mentioned above. 
For this, let $\{\beta_i, i=1,\dots,N_0\}$ and $\{b_j, j=1,\dots,N_1\}$ be 
bases of $X_{0,h}$ and $X_{1,h}$, respectively. With this, we define the 
\emph{residual vectors}
\begin{align}\label{eq:general_residual_vectors}
  \vect{\rho}_0 &:= [r_0(\beta_i)]_{i=1}^{N_0},\\
  \vect{\rho}_1 &:= [r_1(b_j)]_{j=1}^{N_1}.
\end{align}

Hence, we need to bound $\norm{\beta_i}{X_0}$ and $\norm{b_j}{X_1}$ and by 
combining this with \eqref{eq:EstimateIdea} we know how the residual vectors 
will change with $h$. Consequently, one can use these convergence rates to deduce 
if the Galerkin matrices for the different BIOs are implemented correctly by 
employing a manufactured solution to the exterior BVP \eqref{eq:BVP}; calculating 
the traces exactly, and finally letting the code compute the residuals and 
checking whether the observed rate agrees with the expected one.

\begin{remark}
 We can proceed analogously and use the interior Calder\'on identities instead 
 of the exterior ones. Although for the sake of brevity we will not write the
 corresponding computations, it is worth mentioning that the estimates follow
 analogously and that this allows the practitioner to choose how to test their 
 code depending on the analytical or manufactured solutions available. We will 
 show examples regarding both interior and exterior Calder\'on identities when 
 we do numerical experiments in later sections. 
\end{remark}

\section{First Application: Laplacian}
\label{sec:Laplace}

Let us begin by looking for $u$ such that $-\Delta u = 0$ in $\RR^d \setminus 
\overline{\Omega}$ and the decay condition $u(\Vx) = \mathcal{O}(\Vert\Vx^{-1}
\Vert)$ as $\Vert \Vx \Vert \rightarrow\infty$ is verified. This means that its 
weak solution $u\in H^1_{loc}(\Omega)$, and thus, that we consider the Dirichlet 
trace $\tau_D: H^1(\Omega) \to H^{1/2}(\Gamma)$, the Neumann trace $\tau_N: 
H^1_{loc}(\Omega)\to H^{-1/2}(\Gamma)$, and the related BIOs\footnote{We refer 
to \cite[Chapter~2]{Sauter_Schwab} for the definitions of these Sobolev 
spaces.}
\begin{itemize}
\item $\OV: H^{-1/2}(\Gamma)\to H^{1/2}(\Gamma) \quad$ (weakly singular);
\item $\OW: H^{1/2}(\Gamma)\:\:\to H^{-1/2}(\Gamma) \quad$ (hypersingular);
\item $\OK: H^{1/2}(\Gamma)\:\:\to H^{1/2}(\Gamma) \quad$ (double layer);  
\item $\OK': H^{-1/2}(\Gamma)\to H^{-1/2}(\Gamma) \quad$ (adjoint double layer).
\end{itemize}
Hence, we have that, related to \eqref{eq:BVP} and 
\eqref{eq:general_Calderon_identity}, $\mathcal{L} = -\Delta$, $\gamma_0 = 
\tau_D$ and $\gamma_1 = \tau_N$. In particular, the following identities are 
true for the Dirichlet and Neumann trace of $u$ for the exterior BVP:
\begin{align}
\label{eq:calderonLaplaceext}
  \begin{pmatrix}
    \frac{1}{2}\Id +\OK && -\OV\\
    -\OW && \frac{1}{2}\Id - \OK'
  \end{pmatrix} 
  \begin{pmatrix}
  \tau_D u\\ \tau_N u
  \end{pmatrix} =
  \begin{pmatrix}
  \tau_D u\\ \tau_N u
  \end{pmatrix} \iff \begin{pmatrix}
    \frac{1}{2}\Id -\OK && \OV\\
    \OW && \frac{1}{2}\Id + \OK'
  \end{pmatrix} 
  \begin{pmatrix}
  \tau_D u\\ \tau_N u
  \end{pmatrix} = 0.
\end{align}

Next, we proceed to discretize, which requires us to define our boundary element
spaces. 
\begin{definition}[{\cite[Definitions 4.1.36, 4.1.17]{Sauter_Schwab}}]
Let $\Gamma$ be a piecewise smooth surface
and let $\GGG_h$ be a regular surface mesh of $\Gamma$ consisting of triangles. 
Let $\chi=\{\chi_{\tau}:\tau\in\GGG_h\}$ be a mapping vector and let $p\in\NN_0$. 
We first introduce
\begin{equation*}
\mathbb{P}^{\Delta}_p :=\text{span} \{\xi^{\mu}: \mu\in\NN^2_0 \wedge \abs{\mu}
\leq p\}.
\end{equation*}
Then, we define for $p\in\NN_0$
\begin{equation*}
\calS^{p,-1}(\GGG_h):= \{\psi:\Gamma\to \KK: \forall\tau\in\GGG_h, \psi\circ 
\chi_{\tau}\in \mathbb{P}^{\Delta}_p\},
\end{equation*}
and for $p\geq 1$
\begin{equation*}
\calS^{p,0}(\GGG_h):= \{\psi\in C^0(\Gamma): \forall\tau\in\GGG_h, \psi|_{\tau}
\circ\chi_{\tau}\in \mathbb{P}^{\Delta}_p\}.
\end{equation*}
\end{definition}

We remind the reader that we have interpolation operators mapping from subspaces 
of our trace spaces to these BE spaces. Indeed, we write $\OI_h^{p,0}$ for the 
nodal interpolation operator mapping to $\calS^{p,0}(\GGG_h)$, and $\OI_h^{p-1,
-1}$ for the local $L^2$ projector to $\mathcal{S}_{\GGG}^{p-1,-1}$ for $p\geq 
1$. We always let $\GGG_h$ be a shape regular, quasi uniform family of meshes. 

Following the derivation from \eqref{eq:residuals1} and \eqref{eq:residuals2}, 
equation \eqref{eq:calderonLaplaceext} leads to the following 
residuals
\begin{align}
r_D(\psi) :=& \int_{\Gamma}[(\frac{1}{2} \Id -\OK)\OI_h^{p,0} (\tau_D u)
+ \OV(\OI_h^{p-1,-1}\tau_N u)](x) \psi(x) \d S(x) \nonumber\\
=& \int_{\Gamma}[(\frac{1}{2} \Id -\OK)(\OI_h^{p,0} -\Id)(\tau_D u)
+ \OV((\OI_h^{p-1,-1}-\Id)\tau_N u)](x) \psi(x) \d S(x),\label{eq:residual_D} 
\end{align}
for $\psi\in H^{-1/2}(\Gamma)$, and
\begin{align}
r_N(\phi) :=& \int_{\Gamma}[\OW \OI_h^{p,0}(\tau_D u)
+ (\frac{1}{2}\Id +\OK')\OI_h^{p-1,-1}\tau_N u](x) \phi(x) \d S(x)\nonumber\\
=& \int_{\Gamma}[\OW (\OI_h^{p,0}-\Id)(\tau_D u)
+ (\frac{1}{2}\Id +\OK')(\OI_h^{p-1,-1}-\Id)\tau_N u](x) \phi(x) \d S(x), 
\label{eq:residual_N}
\end{align}
for $\phi\in H^{1/2}(\Gamma)$. 

Next step is to bound the residual as in \eqref{eq:EstimateIdea}. To do this we 
need the following results.
\begin{assumption}\label{assumption_Gamma}
The following properties can hold for a surface $\Gamma$.
  \begin{enumerate}
     \item $\Gamma$ is the surface of a polyhedron. Its mesh $\mathcal{G}$ 
     consist of plane panels with straight edges with meshwidth $h>0$.
    \item $\Gamma$ is a curved surface. Then we need certain geometric assumptions, 
    see \cite[Assumption 4.3.18]{Sauter_Schwab}.
  \end{enumerate}
\end{assumption}
\begin{theorem}[{\cite[Theorem 4.3.20]{Sauter_Schwab}}]\label{thm_interpol_discontinuous}
  Let $\Gamma$ fulfill either of the conditions in Assumption~\ref{assumption_Gamma}.
  Then, we have for the interpolation $\OI_h^{p-1,-1}: L^2(\Gamma) \to \mathcal{
  S}_{\GGG}^{p-1,-1}$ and $0\leq t\leq s \leq p$ and all $u\in H^s_{\text{pw}}(\Gamma)$ 
  (where this space is as defined in 
  \cite[Definition 4.1.48]{Sauter_Schwab}) the estimate
  \begin{equation}\label{eq_interpol_discontinuous}
    \|u - \OI_h^{p-1,-1}u\|_{H^{-t}(\Gamma)}\leq C h^{s+t}\|u\|_{H^s(\Gamma)}.
  \end{equation}
\end{theorem}
\begin{theorem}[{\cite[Theorem 4.3.22b]{Sauter_Schwab}}]
Let $\Gamma$ fulfill either of the conditions in Assumption~\ref{assumption_Gamma}. 
Let $u\in H^s(\Gamma)$ for some $1<s\leq p+1$. Then, for any $0\leq t \leq 1$, we 
have
\begin{equation}\label{eq_interpol_continuous}
  \|u \ \OI_h^{p,0}u\|_{H^t(\Gamma)} \leq C h^{s-t}\|u\|_{H^s(\Gamma)}.
\end{equation}
\end{theorem}

\begin{theorem}\label{thm:convergence_Laplace}
Let $q_N^j, j= 1,\dots, N_0$, be a set of nodal basis functions of $\calS^{p-1,
-1}(\GGG_h)$, and let $\varphi_N^i, i = 1,\dots N_1$, be a set of nodal basis 
functions of $\calS^{p,0}(\GGG_h)$. Let $r_D$ and $r_N$ be as introduced in
\eqref{eq:residual_D} and \eqref{eq:residual_N}, and define 
  \begin{equation*}
  \vect{\rho}_D:= [r_D(q_N^j)]_{j=1}^{N_0}, 
  \qquad \vect{\rho}_N:= [r_N(\varphi_N^i)]_{i=1}^{N_1}. 
\end{equation*}
Then, it holds that
\begin{empheq}[box=\fbox]{align*}
\norm{\vect{\rho}_D}{\infty} = \mathcal{O}(h^{p+\frac{1}{2}+\frac{d}{2}}),\quad 
  &\norm{\vect{\rho}_N}{\infty} = \mathcal{O}(h^{p-\frac{1}{2}+\frac{d}{2}}),\\
  \norm{\vect{\rho}_D}{2} = \mathcal{O}(h^{p-\frac{1}{2}+\frac{d}{2}}),\quad &
  \norm{\vect{\rho}_N}{2} = \mathcal{O}(h^{p-\frac{3}{2}+\frac{d}{2}}).
\end{empheq}
\end{theorem}
\begin{proof}
From Equations \eqref{eq:residual_D}, \eqref{eq:residual_N}, we get
\begin{align*}
  |r_D| \leq& [\,\|\frac{1}{2}\Id + \OK\|_{H^{1/2}(\Gamma)\to H^{1/2}(\Gamma)}
  \| (\OI_h^{p,0} -\Id)\tau_D u\|_{H^{1/2}(\Gamma)}\\&
  +\|\OV\|_{H^{-1/2}(\Gamma)\to H^{1/2}(\Gamma)}\|(\OI_h^{p-1,
  -1}-\Id)\tau_N u\|_{H^{-1/2}(\Gamma)}]\|\psi\|_{H^{-1/2}(
  \Gamma)},\\
  |r_N|\leq& [\,\|\OW\|_{H^{1/2}(\Gamma)\to H^{-1/2}(\Gamma)}\|
  (\OI_h^{p,0} -\Id)\tau_D u\|_{H^{1/2}(\Gamma)}\\&+\|\frac{1}{2}\Id 
  -\OK'\|_{H^{-1/2}(\Gamma)\to H^{-1/2}(\Gamma)}\|(\OI_h^{p-1,
  -1}-\Id)\tau_N u\|_{H^{-1/2}(\Gamma)}]\|\phi\|_{H^{1/2}
  (\Gamma)}.
\end{align*}
Using the continuity of the BIOs and Theorems 4.3.20
and 4.3.22b from \cite{Sauter_Schwab}, this can be rewritten as
\begin{align}\label{eq:bound_residual}
  |r_D|\leq c_D h^{p+\frac{1}{2}}\|\psi\|_{H^{-1/2}(\Gamma)},\quad|r_N|
  \leq c_R h^{p+\frac{1}{2}}\|\phi\|_{H^{1/2}(\Gamma)}
\end{align}
for some $h$-independent constants $c_D,c_R>0$.
  
Now we replace $\psi$ with $q_N^j$ and $\phi$ with $\varphi_N^i$. We illustrate 
the procedure for $d=3$. Firstly, 
let $\vect{\pi}\in \GGG_h$ be a panel of diameter $h$. Then, for the reference 
triangle $\hat{K}$, we have some fixed reference shape function $\hat{q}^j$, 
which is the pullback of the basis function $q^j_N$ to $\hat{K}$,
such that
\begin{align*}
\norm{q_N^j}{H^{-1/2}(\Gamma)}^2 \simeq& \int_\Gamma \OV(q^j_N)(x)
q_N^j(x) \d S(x) = 
\int_{\vect{\pi}}\int_{\vect{\pi}}\frac{1}{4\pi \|x-y\|}q_N^j(y)q_N^j(x) 
\d S(x)\d S(y)\\ =& h^{2\cdot 2}\int_{\hat{K}}\int_{\hat{K}} \frac{1}{4
\pi h \|s-t\|}\hat{q}^j(s) \hat{q}^j(t)\d t\d s =\mathcal{O}(h^3),
\end{align*}
where $\simeq$ denotes equality up to a constant, which in this case follows 
from $\OV$ being continuous and elliptic in $H^{-1/2}(\Gamma)$.
Similar arguments work for $d=2$. Hence, we have that
\begin{equation*}
\|q_N^j\|_{H^{-1/2}(\Gamma)}\approx h^{\frac{d}{2}}\quad\text{on }\GGG_h, 
\end{equation*}
with constants independent of $h$.
  Secondly, we have that
\begin{align*}
\|\varphi_N^i\|^2_{H^{1/2}(\Gamma)}\simeq& \int_{\Gamma} (\OW 
\varphi_N^i)(x) \varphi_N^i(x)\d S(x) \\=& 
\int_{\vect{\pi}}\int_{\vect{\pi}}\frac{1}{4\pi \|x-y\|}\curl_{\Gamma} 
\varphi_N^i(y)\cdot\curl_{\Gamma} \varphi_N^i(x) \d S(x)\d S(y).
\end{align*}
Thus, using that the curl scales like $h$ under scaling pullback to the 
reference triangle $\hat{K}$, the same arguments as before yield
\begin{equation*}
\|\varphi_N^i\|_{H^{1/2}(\Gamma)}\approx h^{\frac{d}{2}-1}\quad
\text{on }\GGG_h.
\end{equation*}
Combining this with \eqref{eq:bound_residual} gives the result. 
\end{proof}

\subsection{Implementation and sanity checks}
\label{ssec:LapIR}

We calculate the residuals for some Examples using the library \texttt{Bempp}\footnote{\url{https://bempp.com/}}. 
We use the following solutions to 
$ -\Delta u = 0$ on some domain $\Omega$:
\begin{itemize}
 \item \textbf{Example 1a:} Let $\Omega$ be the unit ball in $\RR^3$, then we 
 consider the exact solution on $\RR^3\setminus\overline{\Omega}$ (in spherical coordinates) 
 $$u(r, \theta, \phi) = r^{-2} e^{i\phi}P_1^1(\cos \theta),$$ 
 where $P_1^1(x) = -(1-x^2)^{1/2} \frac{\d}{\d x}(P_1(x))$, and $P_1(x)$ 
 is the first degree Legendre polynomial. Then the Dirichlet and Neumann traces 
 are, respectively, given by
\begin{align*}
  \tau_D u(r,\theta, \phi) &= r^{-2} e^{i \phi}P_1^1(\cos \theta)|_{\Gamma}\\
  \tau_N u(r,\theta, \phi ) &=  \frac{\partial}{\partial r} u(r,\theta,\phi) 
  |_{\Gamma}= -2 r^{-3}e^{i\phi} P_1^1(\cos \theta)|_{\Gamma}.   
\end{align*}

 \item \textbf{Example 1b:} Let $\Omega$ be the unit ball in $\RR^3$, then we 
 consider the exact solution (in spherical coordinates) 
 $$u(r, \theta, \phi) = r e^{i\phi}P_1^1(\cos \theta),$$ 
 with $P_1^1(x)$ as before. Then the Dirichlet and Neumann traces 
 are, respectively, given by
\begin{align*}
  \tau_D u(r,\theta, \phi) &= r e^{i \phi}P_1^1(\cos \theta)|_{\Gamma}\\
  \tau_N u(r,\theta, \phi ) &=  \frac{\partial}{\partial r} u(r,\theta,\phi) 
  |_{\Gamma}= e^{i\phi} P_1^1(\cos \theta).   
\end{align*}

 \item \textbf{Example 2:} Let $\Omega$ be the unit cube in $\RR^3$. We assume the 
 exact solution is given by 
 $$u(x,y,z) = x^2 - \frac{1}{2}y^2 -\frac{1}{2}z^2.$$ 
Then the Dirichlet and Neumann trace are given by
\begin{equation*}
  \tau_D u = u|_{\Gamma}, \quad \tau_N u = 
  \nabla u|_{\Gamma} \cdot \Vn = \begin{pmatrix}
                                       2x\\-y\\-z
                                      \end{pmatrix}|_{\Gamma}\cdot \Vn .
\end{equation*}
Here, $\Vn$ is the outward pointing normal vector.
\end{itemize}

Our motivation to consider these three examples as ``canonical examples'' is that 
they will give us information about the performance of the Calder\'on residuals 
for exterior problems (Example 1a) vs. interior problems (Examples 1b and 2); and 
extra regular solutions on a sphere (Examples 1a and 1b) vs. standard solution 
on a square (Example 2).

For our implementation, we use the lowest order spaces for all of the 
approximations, i.e., $\calS^{1,0}(\GGG_h)\subset H^{1/2}(\Gamma)$ 
(piecewise linear functions) and $\calS^{0,-1}(\GGG_h)\subset H^{-1/2}
(\Gamma)$ (piecewise constant functions) for some mesh $\GGG_h$. We construct 
the Galerkin matrices for our boundary integral operators as usual.

\subsubsection{Empiric convergence of norms of basis functions}

\begin{table}[h!]
\begin{center}
\begin{tabular}{|c|c|c|}\hline
 \# elements &  $max(diag(\MV))$ & $max(diag(\MW))$\\
 \hline
128   & 1.261e-02 & 2.488e-01\\\hline
 512   & 1.797e-03 & 1.460e-01\\\hline
 2048  & 2.325e-04 & 7.70e-02\\\hline
 8192 & 2.932e-05& 3.904e-02\\\hlinewd{1pt}
observed conv. rate & \green{3.00} & \green{0.92}\\ \hline
expected conv. rate & 3 & 1\\\hline
\end{tabular}
\caption{Maximal element on diagonal of $\MV$ and $\MW$ to numerically 
compute the norm 
of piecewise constant and piecewise linear basis functions.} 
\label{tab:Bempp_pwc_pwl_convergence}
\end{center}
\end{table}

First, we validate the predicted convergence rates of scalar basis functions.
In particular, we compute the $H^{-1/2}(\Gamma)$-norm for piecewise 
constants $q_N^j$, and the $H^{1/2}(\Gamma)$-norm for piecewise linears 
$\varphi_N^i$. Since the weakly singular $\OV$ and hypersingular boundary integral 
operator $\OW$ induce energy norms on $H^{-1/2}(\Gamma)$ and $H^{\frac{1}{2}
}(\Gamma)$, respectively, we calculate $\norm{q_N^j}{H^{-1/2}(\Gamma)}$ 
and $\norm{\varphi_N^i}{H^{1/2}(\Gamma)}$ by inspecting the diagonals 
of their Galerkin matrices $\MV$ and $\MW$. Note that this approach gives a 
measure of the square of the norms of the basis functions, and thus we compare 
them with two times the computed convergence estimate. 
The results are reported in Table~\ref{tab:Bempp_pwc_pwl_convergence}. There, 
we see that the observed convergence 
rates for all of the basis functions are close to the expected ones. Here, 
the observed convergence rates are labelled 'observed conv. rate' 
and were computed using polyfit and taking the logarithm of the values 
reported in the columns.

\subsubsection{Checking convergence of Calder\'on residuals}

Tables~\ref{tab:TestCase2_sphere}, \ref{tab:TestCase2_sphere_interior} and 
\ref{tab:TestCase3_cube} show the obtained results with \texttt{Bempp} for 
Examples 1a, 1b and 2, respectively. 
In all Examples, the observed convergence is larger than expected, but this, 
of course, does not contradict the theory. Moreover, in light of the rates 
from Table~\ref{tab:Bempp_pwc_pwl_convergence}, we have reason to believe this 
arises from the inequalities used to derive the estimate being too crude, and 
not from a mistake in the computed rates. We point out, however, that although 
our estimates are correct, this lack of sharpness may hamper the 
effectivity of Calder\'on residuals as debugging tool, which is what we will 
study in the next section. 

\begin{table}[h!]
\begin{center}
\begin{tabular}{|c|c|c|c|c|}\hline
 \# elements & $\|\Vrho_D\|_{\infty}$ &  $\|\Vrho_D\|_{2}$& $\|\Vrho_N\|_{
 \infty}$& $\|\Vrho_N\|_{2}$\\
 \hline
 128  & 1.0025e-03& 7.6619e-03  & 7.2558e-03  & 3.3546e-02 \\\hline
 512  & 8.0813e-05& 1.1109e-03 & 6.5674e-04& 6.1615e-03\\\hline
 2048 & 5.9248e-06& 1.5188e-04 & 4.2410e-05 & 8.3545e-04 \\\hline
 8192 & 3.8476e-06& 1.9611e-05& 2.6599e-06 & 1.0904e-04\\\hlinewd{1pt}
 observed conv. rate & \green{3.89}& \green{2.95}& \green{3.92} & \green{2.84} \\ \hline
 expected conv. rate &3 &2& 2&1 \\ \hline
\end{tabular}
\caption{Calder\'on residuals 
for Laplace BVP, considering Example 1a on the unit sphere.} 
\label{tab:TestCase2_sphere}
\end{center}
\end{table}

\begin{table}[h!]
\begin{center}
\begin{tabular}{|c|c|c|c|c|}\hline
 \# elements & $\|\Vrho_D\|_{\infty}$ &  $\|\Vrho_D\|_{2}$& $\|\Vrho_N\|_{
 \infty}$& $\|\Vrho_N\|_{2}$\\
 \hline
 128  & 5.4551e-04 & 3.7834e-03& 2.6058e-03 & 1.3876e-02\\\hline
 512  & 4.1249e-05 & 5.3052e-04& 2.1362e-04 & 2.2848e-03\\\hline
 2048 & 2.6771e-06 & 6.8635e-05& 1.5945e-05& 3.1756e-04\\\hline
 8192 & 1.7098e-07 & 8.6937e-06& 1.6783e-06 & 4.2566e-05\\\hlinewd{1pt}
 observed conv. rate & \green{3.99}& \green{3.004}& \green{3.65} & \green{2.87} \\ \hline
 expected conv. rate &3 &2& 2&1 \\ \hline
\end{tabular}
\caption{Calder\'on residuals 
for Laplace BVP, considering Example 1b on the unit sphere.} 
\label{tab:TestCase2_sphere_interior}
\end{center}
\end{table}

\begin{table}[h!]
\begin{center}
\begin{tabular}{|c|c|c|c|c|}\hline
\# elements & $\|\Vrho_D\|_{\infty}$&  $\|\Vrho_D\|_{2}$&$\|\Vrho_N\|_{\infty}$
 &$\|\Vrho_N\|_{2}$\\
 \hline
 84  & 7.2364e-04 & 2.8750e-03 & 8.7829e-04&2.5842e-03\\\hline
 336  & 3.9060e-05 & 2.4859e-04 & 2.3016e-04&9.9100e-04\\\hline
 1344 & 2.7451e-05 & 2.3896e-05 & 2.0798e-05&1.6799e-04\\\hline
 5376 & 1.8921-e06& 2.2553e06 & 2.4550e-06&2.9649e-05\\\hlinewd{1pt}
 observed conv. rate & \green{3.95}& \green{3.43}& \green{2.89} & \green{2.19}\\ \hline
 expected conv. rate &3 & 2 &2 &1 \\ \hline
\end{tabular}
\caption{Calder\'on residuals 
for Laplace BVP, considering Example 2 on the unit cube.} 
\label{tab:TestCase3_cube}
\end{center}
\end{table}

\newpage
\subsection{Numerical results when introducing artificial errors}
\label{ssec:ErrorsLap}

We test how robust our method is by introducing artificial errors in the 
Galerkin matrices of the weakly singular $\OV$ and hypersingular boundary 
integral operator $\OW$. We skip the double layer operators for brevity.

For each test, we calculate the Calder\'on residuals, compare them with the 
standard \emph{method of manufactured solutions (MMS)} convergence test. In 
other words, we do the following convergence test: (i) We build the 
\emph{correct} Galerkin matrices and store them as $\MV, \MK$ and $\MW^m$, 
where $\MW^m$ is the Galerkin matrix corresponding to the modified hypersingular 
boundary integral operator from \cite[Equation 7.26]{Steinbach} with $\overline{
\alpha}=1$. We also build the mass matrix $\MM$ arising from $\Id$; (ii) we 
introduce an artificial error in the computation of the matrices of $\OV$ or 
$\OW$, and store these \emph{defective} Galerkin matrices as $\MV^e$ and 
$\MW^{m,e}$; (iii) We use CG with relative tolerance equal to $10^{-10}$ (or a 
direct solver when CG fails to converge\footnote{Using a direct solver is in 
our case only necessary when the introduced artificial error breaks some 
fundamental 'nice' property of the system matrix, like positive definiteness.}) 
to numerically solve the following linear systems 
\begin{align}
\MV^e \vec{w} &= (\frac{1}{2}\MM + \MK)(\OI_h^{1,0}\tau_D u),\\
\MW^{m,e} \vec{v} &= (\frac{1}{2}\MM - \MK')(\OI_h^{0,-1}\tau_N u),
\end{align}
for the coefficient vectors $\vec{w}$ and $\vec{v}$; (iv) We measure the 
\emph{MMS errors}
\begin{align}
\mathsf{e}_N&:=\|\tau_N u - w_h\|_{H^{-1/2}(\Gamma)}
\approx [(\OI_h^{0,-1}\tau_N u - \vec{w})^{\top} \MV (\OI_h^{0,-1}\tau_N u - 
\vec{w})]^{\frac{1}{2}}, \label{eq:conv_test_1}\\
\mathsf{e}_D&:=\|\tau_D u - v_h\|_{H^{1/2}(\Gamma)}
\:\quad\approx [(\OI_h^{1,0}\tau_D u - \vec{v})^{\top} \widetilde{\MW} (\OI_h^{1,0}
\tau_D u - \vec{v})]^{\frac{1}{2}},\label{eq:conv_test_2}
\end{align}
where $w_h$ is the boundary element function with coefficients $\vec{w}$ and
$v_h$ is the one corresponding to the coefficients $\vec{v}$, i.e. 
$w_h=\sum_i w_i \beta_i$ and  $v_h=\sum_i v_i b_i$;
and $\widetilde{\MW}$ is the Galerkin matrix of the hypersingular 
integral operator for the Helmholtz operator with imaginary wavenumber;
(v) Finally, we repeat steps (i)-(iv) for a sequence of dyadically refined 
meshes and see if this behaves like $h^{1/2}$ for \eqref{eq:conv_test_1} 
and like $h^{3/2}$ for \eqref{eq:conv_test_2}, where $h$ is the meshwidth 
as before.

We make this comparison to contextualize the performance of the Calder\'on 
residuals and assess if the information we get is meaningful.
Since the Galerkin matrices for Example 1a and 1b are the same, and 
Tables~\ref{tab:TestCase2_sphere} and \ref{tab:TestCase2_sphere_interior}
confirm that Calder\'on residuals behave qualitatively the same for an 
exterior BVP and its interior counterpart, we only test Examples 1a and 
2 in this subsection. We report the convergence rate computed with 
linear regression and taking the logarithm of the values reported in the columns. 
This is reported for all quantities in the additional row `ooc'. The last 
row, labelled `eoc' corresponds to the expected convergence rate and is 
provided to make the comparison easier. Given that sometimes $\mathsf{e}_N$ 
and $\mathsf{e}_D$ do not decay monotonically as expected, all tables where 
these errors are measured include the columns labelled `ooc $\mathsf{e}_i$', 
$i=D,N$, which contain the observed convergence rate calculated using 
the error of two consecutive meshsizes. In this way, the unexpected behaviours 
become evident. In the situations where the rate computed with linear regression
and the one calculated between two consecutive meshes disagree, we consider 
the latter to be more relevant.

\subsubsection{Inaccurate quadrature}

We adapt the quadrature for Example 1a and Example 2. Standard parameters in 
\texttt{Bempp} are order 4 for both the singular and the regular quadrature. 
These are adapted to 1 for the singular and 2 for the regular quadrature, 
respectively. Table~\ref{tab:TestCase1a-adapting_quadrature} shows that 
for the exterior extra regular problem (Example 1a), the \emph{MMS} test 
from \eqref{eq:conv_test_1}-\eqref{eq:conv_test_2} gives values in the range of 
machine precision, while the \emph{Calder\'on residual} test obtains observed 
convergence rates that are still close to or even higher than the expected ones. 
For the interior problem (Example 2), we see in 
Table~\ref{tab:TestCase2-adapting_quadrature} that the situation is a bit more 
complicated. For the \emph{MMS} test we have that 
$\mathsf{e}_N$ converges with the expected rate and $\mathsf{e}_D$ does not.
In this case, the \emph{Calder\'on residual} test fails for both $\Vrho_D$ 
and $\Vrho_N$. It is therefore not clear if we can use Calder\'on 
residuals to identify if the quadrature is precise enough. 

\setlength{\tabcolsep}{4pt} 
\begin{table}[H]
\begin{center}\footnotesize
\begin{tabular}{|c|c|c|c|c|c|c|c|c|}\hline
N & $\|\Vrho_D\|_{\infty}$&  $\|\Vrho_D\|_{2}$&$\|\Vrho_N\|_{\infty}$
&$\|\Vrho_N\|_{2}$ & $\mathsf{e}_N$, tol=$10^{-10}$ &ooc $\mathsf{e}_N$
&$\mathsf{e}_D$, tol=$10^{-10}$ & ooc $\mathsf{e}_D$ \\ \hline
128  & 2.090e-02& 1.518e-01  & 1.571e-02& 8.548e-02 &4.784e-08&
&1.322e-07& \\\hline
512  & 3.153e-03& 4.038e-02  & 2.967e-03& 2.587e-02&4.265e-08&  \color{gray}{0.58}$^*$
&9.098e-08& \color{gray}{2.83}$^*$\\\hline
2048 & 4.322e-04& 1.027e-02  & 4.650e-04& 7.388e-03&2.699e-08& \color{gray}{0.48}$^*$
&6.546e-08& \color{gray}{-2.95}$^*$\\\hline
8192 & 5.480e-05& 2.576e-03 &  6.426e-05& 1.968e-03&4.235e-08& \color{gray}{-1.28}$^*$
&1.587e-07& \color{gray}{2.02}$^*$\\\hlinewd{1pt}
ooc & \green{2.94}& \green{2.02}& \green{2.72} & \green{1.86}&\color{gray}{0.12}$^*$ 
&  &\color{gray}{-0.03}$^*$& \\ \hline
eoc  &3 & 2 &2 &1 &0.5&0.5&1.5&1.5\\ \hline
\end{tabular}
\caption{Convergence results with \texttt{Bempp} for Laplace BVP, considering 
Example 1a on the unit sphere and using a less accurate quadrature. \\
$^*$: Note that the values of $\mathsf{e}_N$ and $\mathsf{e}_D$ are in range 
of machine precision, hence the observed rate may be meaningless.} 
\label{tab:TestCase1a-adapting_quadrature}
\begin{tabular}{|c| c|c| c|c|c|c|c|c|}\hline
N & $\|\Vrho_D\|_{\infty}$&  $\|\Vrho_D\|_{2}$&$\|\Vrho_N\|_{\infty}$
&$\|\Vrho_N\|_{2}$ & $\mathsf{e}_N$, tol=$10^{-10}\,$ &ooc $\mathsf{e}_N$
&$\mathsf{e}_D$, tol=$10^{-10}$ & ooc $\mathsf{e}_D$\\ \hline
84  & 3.833e-02 & 1.572e-01  &2.134e-01 &6.977e-01&3.231e-01& &3.544e-01&\\\hline
336  & 1.341e-02& 8.898e-02  & 6.814e-02& 3.937e-01&2.166e-01& 0.58 &2.211e-01&0.68\\\hline
1344 &3.671e-03 & 4.733e-02  & 1.995e-02 &2.084e-01&1.290e-01& 0.75 &1.247e-01& 0.83 \\\hline
5376 & 9.534e-04 & 2.439e-02 & 5.688e-03 & 1.072e-01&7.199e-02& 0.84&6.728e-02&0.90\\\hlinewd{1pt}
ooc &\red{1.79} &\red{0.90} & \red{1.75} & \orange{0.90}&\green{0.72}& 
&\red{0.80}&\\ \hline
eoc  &3 & 2 &2 &1 &0.5&0.5&1.5&1.5\\ \hline
\end{tabular}
\caption{Convergence results with \texttt{Bempp} for Laplace BVP, considering 
Example 2 on the unit cube and using a less accurate quadrature.} 
\label{tab:TestCase2-adapting_quadrature}
\end{center}
\end{table}

\subsubsection{Perturbing entries in $\MV$}
We investigate what happens when we introduce the following artificial errors:
\begin{itemize}
\item \textsf{[Error A]} replacing the diagonal of $\MV$ with random numbers 
between 0 and 10,
\item \textsf{[Error B]} scaling the diagonal of $\MV$ with a factor $\frac{1}{2}$,
\item \textsf{[Error C]} scaling the diagonal of $\MV$ with a factor $10^3$,
\item \textsf{[Error D]} scaling the first half of the diagonal of $\MV$ with 
a factor $10^3$,
\item \textsf{[Error E]} scaling the diagonal of $\MV$ with a factor $h^{-1}$,
where $h$ is the meshwidth.
\end{itemize}
In this section, the residual norms for $\rho_N$ are not calculated, since they 
are not affected when only adapting $\MV$.\\

\noindent\textbf{Error A:}
Table~\ref{tab:ErrorA_V} reports the residuals $\rho_D$ in infinity and Euclidean 
norm when \emph{replacing the diagonal of $\MV$ with random numbers}. They also 
contain the corresponding errors from the MMS test, as described in 
\eqref{eq:conv_test_1}. We see that, in both tables, the Calder\'on residuals 
have convergence rates that are much lower than the predicted rates. But
the MMS test for Example 1a tells us that we do have convergence. The MMS test 
for Example 2 agrees with the convergence of the residuals. Note that this is 
only one instance of replacing the diagonal of $\MV$ with random numbers.\\

\setlength{\tabcolsep}{6pt} 
\begin{table}[h!]
\begin{center}
\begin{tabular}{|c|c|c|c|c|}\hline
\multicolumn{5}{|c|}{Example 1a (unit sphere)}\\ \hline
\# elements & $\|\Vrho_D\|_{\infty}$&  $\|\Vrho_D\|_{2}$& $\mathsf{e}_N$, 
tol=$10^{-10}$ &ooc $\mathsf{e}_N$\\
 \hline
 128  & 2.0714e00& 1.1570e01  &8.3901e-01&  \\\hline
 512  & 2.0233e00& 2.2174e01  &3.0698e-01& 1.45 \\\hline
 2048 & 1.9910e00& 4.2904e01  &2.3534e-01& 0.38\\\hline
 8192 & 1.9948e00& 8.5089e01 &4.4270e-02 & 2.41 \\\hlinewd{1pt}
 ooc& \red{0.02}& \red{-0.99}& \green{1.3}& \\ \hline
 eoc &3 & 2 &0.5  &0.5\\ \hline
\end{tabular}\vspace{0.05in}
\begin{tabular}{|c|c|c|c|c|}\hline
\multicolumn{5}{|c|}{Example 2 (unit cube)}\\ \hline
\# elements & $\|\Vrho_D\|_{\infty}$&  $\|\Vrho_D\|_{2}$& $\mathsf{e}_N$, 
tol=$10^{-10}$ &ooc $\mathsf{e}_N$\\ \hline
84  & 1.9899e00& 5.3918e00  & 8.0563e-01&\\\hline
336  & 1.9840e00& 1.0520e01  & 8.6620e-01& -0.10\\\hline
1344 & 1.9899e00& 2.1601e01  & 8.7636e-01& -0.02\\\hline
5376 & 1.9991e00& 4.2109e01 & 8.7828e-01& -0.003\\\hlinewd{1pt}
ooc & \red{-0.002}& \red{-0.99}&\red{-0.04}& \\ \hline
eoc &3 & 2 &0.5  &0.5\\ \hline
\end{tabular}
\end{center}
\caption{Convergence results with \texttt{Bempp} for one instance of the Laplace 
BVP and introducing \textsf{Error A} (replacing 
diagonal entries of $\MV$ with random numbers). }
\label{tab:ErrorA_V}
\end{table}

\begin{table}[h!]
\begin{center}
\begin{tabular}{|c|c|c|c|c|}\hline
\multicolumn{5}{|c|}{Example 1a (unit sphere)}\\ \hline
\# elements & $\|\Vrho_D\|_{\infty}$&  $\|\Vrho_D\|_{2}$& $\mathsf{e}_N$, direct 
& ooc $\mathsf{e}_N$\\
 \hline
 128  & 1.1163e-02 & 7.8460e-02 &5.0246e-07&  \\\hline
 512  & 1.5458e-03 & 1.9421e-02 & 1.9938e-07& 1.33\\\hline
 2048 & 2.0229e-04 & 4.7858e-03 &1.1956e-07&0.74\\\hline
  8192& 2.5254e-05& 1.1857e-03 &1.6658e-07& -0.48\\\hlinewd{1pt}
 ooc & \green{3.01}& \green{2.07}&\green{0.55}&\\ \hline
 eoc &3 & 2 &0.5  &0.5\\ \hline
\end{tabular}\vspace{0.05in}
\begin{tabular}{|c|c|c|c|c|}\hline
\multicolumn{5}{|c|}{Example 2 (unit cube)}\\ \hline
\# elements & $\|\Vrho_D\|_{\infty}$&  $\|\Vrho_D\|_{2}$& $\mathsf{e}_N$, 
direct &ooc $\mathsf{e}_N$\\
 \hline
  84  & 4.1716e-02& 1.7853e-01  & 4.3362e00& \\\hline
 336  &1.3861e-02 & 9.5410e-02  & 3.0417e00& 0.51\\\hline
 1344 &3.7277e-03 & 4.9033e-2  & 1.9252e00& 0.66\\\hline
 5376 & 9.6080e-04 & 2.4821e-02 & 2.2635e00& -0.23 \\\hlinewd{1pt}
 ooc & \red{1.82} & \red{0.95} & \red{0.35}& \\ \hline
 eoc &3 & 2 &0.5  &0.5\\ \hline
\end{tabular}
\end{center}
\caption{Convergence results with \texttt{Bempp} for Laplace BVP and introducing 
\textsf{Error B} (scaling diagonal of $\MV$ with factor $\frac{1}{2}$). Here, the 
\emph{MMS} test is ran with a direct solver.}
\label{tab:ErrorB_V}
\end{table}

\noindent\textbf{Error B:} Table~\ref{tab:ErrorB_V} shows the results when 
scaling the diagonal of $\MV$ with a factor $\frac{1}{2}$.
In this case, CG does not converge, because the introduced error breaks 
the positive definiteness of the matrix. We refer the reader to Appendix~\ref{app:eig} 
for further details. In order to still test what happens when the nice 
properties of the matrix are broken, we resorted to using a direct solver 
instead. This substantial change in the matrix might also 
influence that the Calder\'on residuals pass the test, but the MMS test 
from \eqref{eq:conv_test_1} gives inconsistent results. From this example, we 
conclude that before implementing any of the two studied convergence tests, the 
practitioner should indeed check that the obtained matrix is symmetric positive 
definite, as we know it should be from the theory.\\

\noindent\textbf{Error C:} Table~\ref{tab:ErrorC_V} displays the results when 
scaling the diagonal of $\MV$ with a factor $10^{3}$. At first glance, there seems 
to be an inconsistency between the Calder\'on residual test and the MMS 
test for Example 1a. However, the values of $\mathsf{e}_N$
are in the range of machine precision. Hence, calculating a rate here may not 
make sense. Indeed, if we consider that having all errors within machine precision 
already means convergence, then both tests agree. For Example 2, we see that again 
the Calder\'on residuals do not detect the error, but this time, the MMS 
really does. One reason why this may be the case, is that the Calder\'on residual 
tests will fail to see errors whose influence decays faster than the convergence rate.
This is the motivation behind \textsf{Error E} later, which seeks to test this 
hypothesis.\\

\begin{table}[h!]
\begin{center}
\begin{tabular}{|c|c|c|c|c|}\hline
\multicolumn{5}{|c|}{Example 1a (unit sphere)}\\ \hline
\# elements & $\|\Vrho_D\|_{\infty}$&  $\|\Vrho_D\|_{2}$& $\mathsf{e}_N$, 
tol=$10^{-10}$ &ooc $\mathsf{e}_N$\\\hline
128  & 2.3134e01  & 1.4726e02 & 2.0481e-08 & \\\hline
512  & 3.0233e00  & 3.7342e01 & 2.9949e-08 & \color{gray}{-0.55}$^*$\\\hline
2048 & 3.9756e-01 & 9.3715e00 & 3.0299e-08 & \color{gray}{-0.02}$^*$ \\\hline
8192 & 4.9968e-02 & 2.3452e00 & 4.2499e-08 &\color{gray}{-0.50}$^*$ \\\hlinewd{1pt}
ooc & \green{3.03}& \green{2.05}& \color{gray}{-0.33}$^*$&\\ \hline
eoc &3 & 2 &0.5  &0.5\\ \hline
\end{tabular}\vspace{0.05in}
\begin{tabular}{|c|c|c|c|c|}\hline
\multicolumn{5}{|c|}{Example 2 (unit cube)}\\ \hline
\# elements & $\|\Vrho_D\|_{\infty}$&  $\|\Vrho_D\|_{2}$& $\mathsf{e}_N$, 
tol=$10^{-10}$ &ooc $\mathsf{e}_N$\\\hline
84   & 1.3152e01 & 4,2215e01  & 8.7817e-01 & \\\hline
336  & 1.6526e00 & 1.0596e01  & 8.7556e-01 & 0.004 \\\hline
1344 &2.0858e-01 & 2.6702e00  & 8.7045e-01 & 0.01 \\\hline
5376 &2.6567e-02 & 6.7815e-01 & 8.6049e-01 & 0.02\\\hlinewd{1pt}
ooc &\green{2.98} & \green{1.97}&\red{0.01}& 
\\ \hline
eoc&3 & 2 &0.5  &0.5\\ \hline
\end{tabular}
\end{center}
\caption{Convergence results with \texttt{Bempp} for Laplace BVP and introducing 
\textsf{Error C} (scaling diagonal of $\MV$ with factor $10^3$).\\
$^*$: Note that the values of $\mathsf{e}_N$ for Example 1a are in range 
of machine precision, hence the observed rate may be meaningless.}
\label{tab:ErrorC_V}
\end{table}

\noindent\textbf{Error D:} Table~\ref{tab:ErrorD_V} summarizes the results when 
scaling the first half of the diagonal of $\MV$ with a factor $10^{3}$. We see 
in both examples that the convergence rate of Calder\'on residuals is 
approximately equal to the expected one, and that the MMS test tells us 
that we do not have convergence. The observed 
behaviour is similar to what we obtained for \textsf{Error C}. We will check in 
the next experiment if we can improve the agreement between the two tests by 
introducing an error that is $h$-dependent.\\

\begin{table}[h!]
\begin{center}
\begin{tabular}{|c|c|c|c|c|}\hline
\multicolumn{5}{|c|}{Example 1a (unit sphere)}\\ \hline
\# elements & $\|\Vrho_D\|_{\infty}$&  $\|\Vrho_D\|_{2}$& $\mathsf{e}_N$, 
tol=$10^{-10}$ &ooc $\mathsf{e}_N$\\ \hline
128  & 2.3135e01  & 1.0413e02 & 9.8562e-09 &  \\\hline
512  & 3.0233e00  & 2.6404e01 & 5.9872e-08 & \color{gray}{-2.60}$^*$ \\\hline
2048 & 3.9756e-01 & 6.2667e00 & 4.7220e-08 &  \color{gray}{0.34}$^*$\\\hline
8192 & 4.9968e-02 & 1.6583e00 & 4.2289e-08 & \color{gray}{0.16}$^*$ \\\hlinewd{1pt}
ooc & \orange{2.95}& \green{2.0}& \color{gray}{-0.6}$^*$& 
\\ \hline
eoc &3 & 2 &0.5  &0.5\\\hline
\end{tabular}\vspace{0.05in}
\begin{tabular}{|c|c|c|c|c|}\hline
\multicolumn{5}{|c|}{Example 2 (unit cube)}\\ \hline
\# elements & $\|\Vrho_D\|_{\infty}$&  $\|\Vrho_D\|_{2}$& $\mathsf{e}_N$, 
tol=$10^{-10}$ &ooc $\mathsf{e}_N$\\ \hline
84   & 1.3104e01  & 3.4270e01  & 7.5909e-01 & \\\hline
336  & 1.6380e00  & 8.5672e00  & 7.4645e-01 & 0.02 \\\hline
1344 & 2.0475e-01 & 2.1418e00  &7.3703e-01 & 0.02 \\\hline
5376 & 2.5593e-02 & 5.3545e-01 & 7.2576e-01 & 0.02\\\hlinewd{1pt}
ooc &\green{3.00} &\green{2.00} & \red{0.02}& 
\\ \hline
eoc &3 & 2 &0.5  &0.5\\ \hline
\end{tabular}
\end{center}
\caption{Convergence results with \texttt{Bempp} for Laplace BVP and introducing 
\textsf{Error D} (scaling first half of the diagonal of $\MV$ with factor $10^3$).\\
$^*$: Note that the values of $\mathsf{e}_N$ for Example 1a are in range 
of machine precision, hence the observed rate may be meaningless.}
\label{tab:ErrorD_V}
\end{table}

\begin{table}[h!]
\begin{center}
\begin{tabular}{|c|c|c|c|c|}\hline
\multicolumn{5}{|c|}{Example 1a (unit sphere)}\\ \hline
\# elements & $\|\Vrho_D\|_{\infty}$&  $\|\Vrho_D\|_{2}$& $\mathsf{e}_N$, 
tol=$10^{-10}$ &ooc $\mathsf{e}_N$\\ \hline
 128  & 2.9101e-02 & 1.7769e-01 & 2.0110e-08&  \\\hline
 512  & 1.2081e-02 & 1.4881e-01 & 4.2862e-08& \color{gray}{-1.09}$^*$ \\\hline
 2048 & 3.8291e-03 & 9.0247e-02 & 4.3826e-08& \color{gray}{ -0.03}$^*$\\\hline
 8192 & 1.0467e-03& 4.9126e-02 & 6.0315e-08& \color{gray}{-0.46}$^*$\\\hlinewd{1pt}
 ooc& \red{1.65}& \red{0.65}& \color{gray}{-0.49}$^*$& 
 \\ \hline
 eoc &3 & 2 &0.5  &0.5\\ \hline
\end{tabular}\vspace{0.05in}
\begin{tabular}{|c|c|c|c|c|}\hline
\multicolumn{5}{|c|}{Example 2 (unit cube)}\\ \hline
\# elements & $\|\Vrho_D\|_{\infty}$&  $\|\Vrho_D\|_{2}$& $\mathsf{e}_N$, 
tol=$10^{-10}$ &ooc $\mathsf{e}_N$\\ \hline
  84  &6.6493e-02 & 2.4767e-01  & 3.4313e-01&  \\\hline
 336  & 2.0875e-02& 1.3496e-01  &3.8876e-01&-0.18 \\\hline
 1344 &5.5837e-03 &7.0075e-02   &4.0989e-01& -0.008 \\\hline
 5376 & 1.4376e-03& 3.5667e-02 & 4.2018e-01& -0.04\\\hlinewd{1pt}
 ooc & \red{1.85}& \red{0.93}&\red{-0.10}& 
 \\ \hline
 eoc &3 & 2 &0.5  &0.5\\ \hline
\end{tabular}
\end{center}
\caption{Convergence results with \texttt{Bempp} for Laplace BVP and introducing 
\textsf{Error E} (scaling diagonal of $\MV$ with factor $h^{-1}$).\\
$^*$: Note that the values of $\mathsf{e}_N$ for Example 1a are in range 
of machine precision, hence the observed rate may be meaningless.}
\label{tab:ErrorE_V}
\end{table}

\newpage
\noindent\textbf{Error E:}
Table~\ref{tab:ErrorE_V} shows the results when scaling the diagonal of $\OV$ 
with a factor $h^{-1}$. We see that in both Examples the convergence rate 
of the Calder\'on residuals is below the expected one, and that the energy 
norm test tells us that we do not expect convergence. But, as we have seen before, 
the values of  $\mathsf{e}_N$ for Example 1a are again in the range of machine 
precision, which means that we have convergence. Hence, for Example 2, the 
two tests agree and are able to detect that there is a problem, but for 
Example 1a the tests disagree. 

\subsubsection{Perturbing entries in $\MW$}
We investigate what happens when we introduce the same artificial errors as 
before but now for the hypersingular operator $\OW$, i.e.,:
\begin{itemize}
\item \textsf{[Error A]} replacing the diagonal of $\MW$ with random numbers 
between 0 and 10,
\item \textsf{[Error B]} scaling the diagonal of $\MW$ with a factor $\frac{1}{2}$,
\item \textsf{[Error C]} scaling the diagonal of $\MW$ with a factor $10^3$,
\item \textsf{[Error D]} scaling the first half of the diagonal of $\MW$ with 
a factor $10^3$,
\item \textsf{[Error E]} scaling the diagonal of $\MW$ with a factor $h^{-1}$,
where $h$ is the meshwidth.
\end{itemize}
In this section, the residual norms for $\rho_D$ are not calculated, since they 
are not affected when adapting $\MW$. 

\begin{table}[h!]
\begin{center}
\begin{tabular}{|c|c|c|c|c|}\hline
\multicolumn{5}{|c|}{Example 1a (unit sphere)}\\ \hline
\# elements & $\|\Vrho_N\|_{\infty}$&  $\|\Vrho_N\|_{2}$& $\mathsf{e}_D$, 
tol=$10^{-10}$& ooc $\mathsf{e}_D$\\ \hline
128  &7.5889e-01 & 2.8135e00 & 3.4701e01&  \\\hline
512  & 8.9943e-01 & 6.3053e00& 1.1577e01&  1.58\\\hline
2048 & 9.5124e-01 & 1.3971e01 & 1.7186e03& -7.21 \\\hline
8192 & 9.7547e-01 & 2.854e01 & 6.2075e00& 8.11 \\\hlinewd{1pt}
ooc & \red{-0.12}& \red{-1.15}& \red{0.04}&  \\
\hline
 eoc &2 & 1 &1.5&1.5 \\ \hline
\end{tabular}\vspace{0.05in}
\begin{tabular}{|c|c|c|c|c|}\hline
\multicolumn{5}{|c|}{Example 2 (unit cube)}\\ \hline
\# elements & $\|\Vrho_N\|_{\infty}$&  $\|\Vrho_N\|_{2}$& $\mathsf{e}_D$,  
tol=$10^{-10}$ & ooc $\mathsf{e}_D$ \\ \hline
84  &6.8970e-01 & 1.2594e00& 2.2994e00& \\\hline
336  &8.6444e-01  & 3.1304e00 & 7.9503e00&-1.79 \\\hline
1344 & 8.9094e-01 &6.6847e00 & 7.4968e00&0.08\\\hline
5376 & 9.2668e-01 & 1.3733e01 & 1.3001e01& -0.79\\\hlinewd{1pt}
ooc & \red{-0.13}&\red{-1.1} & \red{-0.7}& 
\\   \hline
 eoc &2 & 1 &1.5&1.5 \\ \hline
\end{tabular}
\end{center}
\caption{Convergence results with \texttt{Bempp} for Laplace BVP and introducing 
\textsf{Error A} (replacing diagonal of $\MW$ with random numbers). }
\label{tab:ErrorA_W}
\end{table}

In Tables~\ref{tab:ErrorA_W}-\ref{tab:ErrorE_W} we see that 
the observed convergence rates for Calder\'on residuals are lower than the 
expected ones, which means the Calder\'on residual test detects all the introduced 
errors that we tried. We see in Tables \ref{tab:ErrorA_W} and \ref{tab:ErrorB_W} 
that the convergence rates for  $\mathsf{e}_D$ in Example 1a are fluctuating. 
This might be because in Table \ref{tab:ErrorA_W}, the introduced error is 
done with random numbers, and it could happen that we get a 'worse' matrix 
for some realizations of this. In Table \ref{tab:ErrorB_W} we have the same 
problem as when we are introducing \textsf{Error B} for $\MV$ in Table 
\ref{tab:ErrorB_V}, where scaling the diagonal of the matrix destroys the 
good properties of the matrix. Furthermore, in Tables \ref{tab:ErrorC_W} 
and \ref{tab:ErrorE_W} for Example 1a, we see that $\mathsf{e}_D$ has 
values in the range of machine precision, so that calculating a 
convergence rate does not give us much information. In all other cases, 
the convergence rate of $\mathsf{e}_D$ is lower than expected, thus 
matching the predictions from the residuals.

\begin{table}[h!]
\begin{center}
\begin{tabular}{|c|c|c|c|c|}\hline
\multicolumn{5}{|c|}{Example 1a (unit sphere)}\\ \hline
\# elements & $\|\Vrho_N\|_{\infty}$&  $\|\Vrho_N\|_{2}$&   $\mathsf{e}_D$,  
tol=$10^{-10}$& ooc $\mathsf{e}_D$\\ \hline
128  & 1.3139e-01 & 8.1555e-01 &1.6947e-06&  \\\hline
512  & 6.7798e-02& 8.0597e-01 & 1.3986e-07& 3.84 \\\hline
2048 &  3.4590e-02& 8.0308e-01 & 6.4068e-07&-2.23\\\hline
8192& 1.7440e-02 & 8.0233e-01 &2.6949e-07& 1.25 \\\hlinewd{1pt}
ooc & \red{1.0} & \red{0.01} & \red{0.58}&  
\\ \hline
eoc &2 & 1 &1.5&1.5 \\ \hline
\end{tabular}\vspace{0.05in}
\begin{tabular}{|c|c|c|c|c|}\hline
\multicolumn{5}{|c|}{Example 2 (unit cube)}\\ \hline
\# elements & $\|\Vrho_N\|_{\infty}$&  $\|\Vrho_N\|_{2}$&  $\mathsf{e}_D$,  
tol=$10^{-10}$& ooc $\mathsf{e}_D$\\ \hline
84  &8.5418e-02 & 2.8696e-01& 3.1694e01&\\\hline
336  &5.2286e-02  & 3.0612e-01 & 3.8364e00&3.05\\\hline
1344 & 2.6438e-02 &3.1681e-01 & 3.4635e00&0.15\\\hline
5376 & 1.3784e-02 & 3.2257e-01 & 1.9136e00&0.86\\\hlinewd{1pt}
ooc &\red{0.89} & \red{-0.06}&\red{1.22}& 
\\\hline
 eoc &2 & 1 &1.5&1.5 \\ \hline
\end{tabular}
\end{center}
\caption{Convergence results with \texttt{Bempp} for Laplace BVP and introducing 
\textsf{Error B} (scaling diagonal of $\MW$ with factor $\frac{1}{2}$). }
\label{tab:ErrorB_W}
\end{table}

\begin{table}[h!]
\begin{center}
\begin{tabular}{|c|c|c|c|c|}\hline
\multicolumn{5}{|c|}{Example 1a (unit sphere)}\\ \hline
\# elements & $\|\Vrho_N\|_{\infty}$&  $\|\Vrho_N\|_{2}$&  $\mathsf{e}_D$,  
tol=$10^{-10}$& ooc $\mathsf{e}_D$ \\ \hline
 128  & 2.6579e02 & 1.6463e03 & 9.6177e-09  &\\\hline
 512  & 1.3648e02 & 1.6128e03 & 3.7631e-08&\color{gray}{-2.10}$^*$\\\hline
 2048 & 6.9163e01& 1.6049e03 & 3.6792e-08& \color{gray}{0.03}$^*$\\\hline
 8192 & 3.4849e01& 1.6031e03 & 2.1073e-08&\color{gray}{0.81}$^*$\\\hlinewd{1pt}
 ooc & \red{1.00}& \red{0.01} & \color{gray}{-0.34}$^*$& \\ \hline
 eoc &2 & 1 &1.5 &1.5\\ \hline
\end{tabular}\vspace{0.05in}
\begin{tabular}{|c|c|c|c|c|}\hline
\multicolumn{5}{|c|}{Example 2 (unit cube)}\\ \hline
\# elements & $\|\Vrho_N\|_{\infty}$&  $\|\Vrho_N\|_{2}$& $\mathsf{e}_D$, 
tol=$10^{-10}$& ooc $\mathsf{e}_D$\\ \hline
 84  &1.7183e02 & 5.7502e02 & 1.1518e00 &\\\hline
 336  & 1.0475e02 & 6.1184e02 & 1.1538e00&-0.003\\\hline
 1344 & 5.2853e01 & 6.3300e02 & 1.1540e00 &-0.0003\\\hline
 5376 & 2.7539e01 & 6.4449e02 & 1.1540e00&0\\\hlinewd{1pt}
 ooc & \red{0.89} & \red{-0.05} & \red{0.00}&\\\hline
 eoc &2 & 1 &1.5 &1.5 \\ \hline
\end{tabular}
\end{center}
\caption{Convergence results with \texttt{Bempp} for Laplace BVP and introducing 
\textsf{Error C} (scaling diagonal of $\MW$ with factor $10^3$). \\
$^*$: Note that the values of $\mathsf{e}_D$ for Example 1a are in range 
of machine precision, hence the observed rate may be meaningless.
}
\label{tab:ErrorC_W}
\end{table}

\begin{table}[h!]
\begin{center}
\begin{tabular}{|c|c|c|c|c|}\hline
\multicolumn{5}{|c|}{Example 1a (unit sphere)}\\ \hline
\# elements & $\|\Vrho_N\|_{\infty}$&  $\|\Vrho_N\|_{2}$&  $\mathsf{e}_D$,  
tol=$10^{-10}$& ooc$\mathsf{e}_D$\\ \hline
 128  &2.6579e02& 1.1272e03 & 1.0095e00 & \\\hline
 512  &1.3648e02 & 1.1320e03 & 9.986e-01& 0.02\\\hline
 2048 &6.9163e01 & 1.1207e03 & 8.966e-01& 0.16\\\hline
 8192 &3.4895e01 &1.1161e03 & 6.266e-01&0.52\\\hlinewd{1pt}
 ooc &\red{1.00} & \red{0.01}& \red{0.23}& \\\hline
 eoc &2 & 1 &1.5&1.5\\ \hline
\end{tabular}\vspace{0.05in}
\begin{tabular}{|c|c|c|c|c|}\hline
\multicolumn{5}{|c|}{Example 2 (unit cube)}\\ \hline
\# elements & $\|\Vrho_N\|_{\infty}$&  $\|\Vrho_N\|_{2}$& $\mathsf{e}_D$,  
tol=$10^{-10}$&ooc $\mathsf{e}_D$\\ \hline
 84  & 1.5644e02& 4.0097e03& 9.4962e-01 & \\\hline
 336  & 9.2386e01 & 4.1266e03 & 9.5179e-01&-0.03 \\\hline
 1344 & 5.2853e01 & 4.5000e03& 9.9276e-01& -0.06\\\hline
 5376 & 2.7539e01 & 4.7043e03 & 1.0096e00& -0.02\\\hlinewd{1pt}
 ooc & \red{0.83}& \red{-0.08} &\red{-0.03}& 
 \\\hline
 eoc &2 & 1 &1.5 &1.5\\ \hline
\end{tabular}
\end{center}
\caption{Convergence results with \texttt{Bempp} for Laplace BVP and introducing 
\textsf{Error D} (scaling first half of diagonal of $\MW$ with factor $10^3$). }
\label{tab:ErrorD_W}
\end{table}

\begin{table}[t]
\begin{center}
\begin{tabular}{|c|c|c|c|c|}\hline
\multicolumn{5}{|c|}{Example 1a (unit sphere)}\\ \hline
\# elements & $\|\Vrho_N\|_{\infty}$&  $\|\Vrho_N\|_{2}$& $\mathsf{e}_D$,  
tol=$10^{-10}$& ooc $\mathsf{e}_D$\\ \hline
 128  & 3.3031e-01 & 2.0450e00 & 2.4194e-08&\\\hline
 512  & 5.4697e-01 & 6.4588e00 & 9.5160e-09& \color{gray}{1.44}$^*$\\\hline
 2048 & 6.6674e-01 & 1.5471e01 & 3.0787e-08& \color{gray}{-1.72}$^*$\\\hline
 8192 & 7.3019e-01 & 3.3589e01 & 2.6546e-08&\color{gray}{0.21}$^*$\\\hlinewd{1pt}
 ooc & \red{-0.38} & \red{-1.4} & \color{gray}{-0.22}$^*$ &
 \\ \hline
 eoc &2 & 1 &1.5&1.5 \\ \hline
\end{tabular}\vspace{0.05in}
\begin{tabular}{|c|c|c|c|c|}\hline
\multicolumn{5}{|c|}{Example 2 (unit cube)}\\ \hline
\# elements & $\|\Vrho_N\|_{\infty}$&  $\|\Vrho_N\|_{2}$& $\mathsf{e}_D$,  
tol=$10^{-10}$& ooc $\mathsf{e}_D$\\ \hline
 84  & 2.3949e-01& 8.0032e-01& 8.080e-01& \\\hline
 336  & 3.9627e-01 & 2.3139e00& 1.0560e00& -0.39\\\hline
 1344 & 4.5266e-01 & 5.4212e00& 1.1301e00&-0.10\\\hline
 5376 & 4.9928e-01 & 1.1684e01 & 1.1483e00&-0.02\\\hlinewd{1pt}
 ooc &\red{-0.34} &\red{-1.28} & \red{-0.16}&
 \\\hline
 eoc &2 & 1 &1.5&1.5 \\ \hline
\end{tabular}
\end{center}
\caption{Convergence results with \texttt{Bempp} for Laplace BVP and introducing 
\textsf{Error E} (scaling diagonal of $\MW$ with factor $h^{-1}$).\\ 
$^*$: Note that the values of $\mathsf{e}_D$ for Example 1a are in range 
of machine precision, hence the observed rate may be meaningless.}
\label{tab:ErrorE_W}
\end{table}

\clearpage
\section{Second Application: Time-Harmonic Maxwell's Equations}
\label{ssec:Maxwell}

Let $k>0$, we solve 
\begin{equation*}
  \bcurl\;\bcurl\; u - k^2 u = 0 \quad \text{ in } \RR^3 \setminus 
  \overline{\Omega},
\end{equation*}
with Silver-M\"uller radiation conditions. We will again start by considering 
what are the related functional spaces, for which we use the notation and 
definitions from \cite[Sect.~2]{BuH03}. We will, however, be specific in the 
definition of the traces and BIOs, as there are different conventions in 
the literature and this will slightly change the concrete formula for the 
residuals, but obviously, the idea remains the same. 

The weak solution of this exterior BVP $u\in \BH_{loc}(\curl,\Omega)$, and we 
consider the tangential trace (sometimes referred as \emph{twisted} tangential 
trace) $\gamma_{\times} \,:\,  \BH(\bcurl,\Omega)\to \Hdiv$ and the magnetic trace 
$\gamma_R: \BH_{loc}(\bcurl^2,\R^3\setminus{\Gamma})\to \Hdiv$ defined as
\begin{align*}
 \gamma_{\times} u := u|_\Gamma \times \Vn, \quad \text{ and } \quad \gamma_R u := 
 \gamma_{\times}(\bcurl u),
\end{align*}
respectively, where the spaces are as defined in \cite[Section 2]{BuH03}. Moreover, the related BIOs are 
\begin{align*}
\OE_k := \{\gamma_{\times}\}_{\Gamma}\circ \SLk, \qquad \OH_k := \{\gamma_{\times}\}_{\Gamma}
\circ \DLk,
\end{align*}
where $\SLk$ and $\DLk$ are the single and double layer potentials as introduced 
in \cite[Sect.~2 and 4, Eq.(27)-(28)]{BuH03}. We remark that $\OE_k$ and $\OH_k$ 
are continuous mappings from $\Hdiv$ to $\Hdiv$ when considering the 
\emph{anti-symmetric} duality pairing 
\begin{align}
 \dual{\Vu}{\Vv}_{\times, \Gamma} := \int_{\Gamma} (\Vu\times \Vn) \cdot \Vv dS, 
 \quad \Vu,\Vv \in \mathbf{L}^2_t(\Gamma),
\end{align}
but they map from $\Hdiv$ to $\Hcurl$ when considering the \emph{symmetric} 
duality pairing 
\begin{align}
 \dual{\Vu}{\Vv}_{\Gamma} := \int_{\Gamma} \Vu \cdot \Vv dS, \quad 
 \Vu,\Vv \in \mathbf{L}^2_t(\Gamma).
\end{align}
In the following, we will consider the symmetric duality pairing because that 
is what agrees with the implementations in \texttt{Bempp}. 
In order to test codes following other conventions, one just has to rotate 
the functions when needed, as will be done in the derivations later. However, 
as a consequence of the following Lemma, the rotations do not affect the 
computed convergence rates. 
\begin{lemma}
\label{lemma:equiv_hcurl_hdiv}
For $v\in \Hcurl$ and $\Vn$ the outward pointing normal vector, we have that
\begin{equation}\label{eq:equivalence_curl_div}
    \|v\|_{\Hcurl} = \|\Vn\times v\|_{\Hdiv}.
\end{equation}
\end{lemma}
The proof can be found in Appendix~\ref{app:Lemma}.

Now, let us return to the problem at hand.
The corresponding Calderón identity for the exterior BVP is given by (see 
\cite[Equation (34)]{BuH03})
\begin{align}
\label{eq:calderon}
  \begin{pmatrix}
    \frac{1}{2}\Id-\OH_k & -\OE_k\\
    -\OE_k & \frac{1}{2}\Id - \OH_k
  \end{pmatrix} \begin{pmatrix}
                  \gamma_{\times} u\\ 
                  \gamma_R u
                \end{pmatrix} = \begin{pmatrix}
                  \gamma_{\times} u\\ 
                  \gamma_R u
                \end{pmatrix} \iff \begin{pmatrix}
    \frac{1}{2}\Id+\OH_k & \OE_k\\
    \OE_k & \frac{1}{2}\Id +\OH_k
  \end{pmatrix} \begin{pmatrix}
                  \gamma_{\times} u\\ 
                  \gamma_R u
  \end{pmatrix} =0.
\end{align}

Next, we proceed to discretize. We illustrate the procedure using Raviart-Thomas 
basis functions, and will comment on the difference when using Rao–Wilson–Glisson 
(RWG) basis functions afterwards. 

Let $\GGG_h$ be a given mesh of $\Gamma$. For $\Hdiv$ we take the div-conforming 
lowest order Raviart-Thomas edge elements $\mathcal{RT}^0(\GGG_h)$ as defined 
in \cite{RaT77}, while for $\Hcurl$ we consider the curl-conforming lowest 
order N\'ed\'elec edge elements $\mathcal{N}^0(\GGG_h)$ as introduced in 
\cite[Sect.~7.2]{Hip07}.

Let $\Ph\,: \Hdiv \to \mathcal{RT}^0(\GGG_h)$ is the orthogonal projection 
from \cite[Theorem~14]{BuH03}. Then, we have the following Calder\'on residuals
\begin{align}
  r_{M,1}(\psi) &= \int_{\Gamma} [\OE_k \Ph (\gamma_R u) \nonumber
  +(\frac{1}{2}\Id +\OH_k)\Ph (\gamma_{\times} u)]\psi\; ds(x)  \\
  &= \int_{\Gamma} [\OE_k (\Ph - \Id) (\gamma_R u) 
   +(\frac{1}{2}\Id +\OH_k)(\Ph-\Id) (\gamma_{\times} u)]\psi\; ds(x),  
  \label{eq:residual_1}\\
  r_{M,2}(\psi) &= \int_{\Gamma} [\OE_k \Ph (\gamma_{\times} u) 
  +(\frac{1}{2}\Id +\OH_k)\Ph (\gamma_R u)]\psi\; ds(x)\nonumber \\
  &= \int_{\Gamma} [\OE_k (\Ph - \Id) (\gamma_{\times} u) 
   +(\frac{1}{2}\Id +\OH_k)(\Ph-\Id) (\gamma_R u)]\psi\; ds(x)  
  \label{eq:residual_2}
\end{align}
for $\psi \in \Hcurl$.

Now we can state the desired convergence result.
\begin{theorem}
\label{thm:convergence_maxwell}
Let $\Omega_h$ be a tetrahedral mesh of $\Omega$ and let $\GGG_h$ be a triangular 
surface mesh of $\Gamma$ such that $\GGG_h = \Omega_h|_{\Gamma}$ up to geometric 
approximation of the boundary.
Let $\beta_e, e=1,\dots,N$ be a basis of $\mathcal{N}^0(\GGG_h)$. Let $r_{M,1}$
and $r_{M,2}$ be as introduced in \eqref{eq:residual_1} and \eqref{eq:residual_2}, 
and define
\begin{align*}
\vect{\rho}_{M,1}:= [r_{M,1}(\beta_e)]_{e=1}^{N},\quad \vect{\rho}_{M,2}:= 
[r_{M,2}(\beta_e)]_{e=1}^{N}.
\end{align*}
If the solution $u$ of our exterior BVP is such that $\gamma_R u, 
\gamma_{\times} u\in \BH^1_{\times}(\div_{\Gamma},\Gamma)$, then we have
\begin{empheq}[box=\fbox]{align*}
\|\vect{\rho}_{M,1}\|_{\infty} = \mathcal{O}(h^1),\quad &\|\vect{\rho}_{M,2}
\|_{\infty}=\mathcal{O}(h^1),\\
\|\vect{\rho}_{M,1}\|_{2} = \mathcal{O}(1),\quad &\|\vect{\rho}_{M,2}\|_{2}
=\mathcal{O}(1).
\end{empheq}
\end{theorem}
\begin{proof}
We only show the result for $\vect{\rho}_{M,1}$, as $\vect{\rho}_{M,2}$ can be 
obtained analogously. From \eqref{eq:residual_1}, it follows that we have
\begin{align*}
|r_{M,1}|\leq &(\|\OE_k\|_{\Hdiv \to \Hdiv} \|(\Ph -\Id)\gamma_R u\|_{\Hdiv}\\
&+\|\frac{1}{2}\Id +\OH_k\|_{\Hdiv\to\Hdiv}\|(\Ph -\Id)\gamma_{\times} u\|_{\Hdiv})
\|\psi\|_{\Hcurl}
\end{align*}
By continuity and \cite[Thm 14]{BuH03} we get 
\begin{equation}\label{eq:bound_r_M}
|r_{M,1}|\leq c_1 h^{\frac{3}{2}}\|\psi\|_{\Hcurl}
\end{equation}
for some constant $c_1>0$. We are interested in the decay rate when we take 
$\psi = \beta_e$, where $\beta_e$ is a N\'ed\'elec basis function. 
We have that 
$\tilde{\beta_e}=\Vn\times \beta_e$, where $\tilde{\beta}_e$ is a Raviart-Thomas 
basis function. For a domain basis function $b_e$ in $\mathcal{ND}_1(\Omega_h)$ 
(as defined in \cite{HiM12}, with $\tilde{\beta_e} = \gamma_{\times}(b_e)$), it holds 
that $\|\tilde{\beta_e}\|_{\Hdiv}\sim \|b_e\|_{H(\bcurl,\Omega)} $, since
\begin{equation*}
\|\gamma_{\times} b_e\|_{\Hdiv} = \|\tilde{\beta}_e\|_{\Hdiv} 
\leq \|b_e\|_{H(\bcurl,\Omega)},
\end{equation*}
where we used the continuity of $\gamma_{\times}$ in the last inequality.
The other direction follows from \cite[Lemma 5.7]{HiM12}. This equivalence is 
useful, since from Lemma \ref{lemma:equiv_hcurl_hdiv} it holds that $\|\beta_e
\|_{\Hcurl} \sim \|\tilde{\beta}_e\|_{\Hdiv}$ and from \cite{HiM12} 
we know that $\|b_e\|_{H(\bcurl,\Omega)}\leq c_2 h^{-1/2}$ for some 
constant $c_2$. Combining this with \eqref{eq:bound_r_M} gives the result. 
\end{proof}

\begin{remark}
\label{rem:scaling}
Instead of using Raviart-Thomas and N\'ed\'elec basis functions in 
Theorem~\ref{thm:convergence_maxwell}, we could also have chosen RWG and 
\emph{scaled} N\'ed\'elec (SNC) basis functions. Let us denote, for simplicity, 
the basis functions of the Raviart-Thomas, RWG, N\'ed\'elec and scaled 
N\'ed\'elec function spaces by $RT_i$, $RWG_i$, $NC_i$ and $SNC_i$. 
We remind the reader that the following relation holds between 
\cite[Equation (17) and (18)]{SBBSW17}
\begin{align*}
  RWG_i(r) &:= l_i \, RT_i(r),\\
  SNC_i(r) &:= l_i \, NC_i(r),
\end{align*}
where $l_i$ is the length of the edge where the basis function is 
defined. 

This scaling factor means that when using RWG and SNC basis functions, we 
expect the convergence of the residuals to be of \emph{one order higher} than 
what we obtained in Theorem \ref{thm:convergence_maxwell}. 
\end{remark}

\subsection{Implementation and sanity checks}
\label{sssec:MaxwellIR}

We use the following solutions to our exterior Maxwell BVP when $\Omega$ is 
the unit cube in $\RR^3$:
\begin{itemize}
 \item \textbf{Example 3:} We consider the exact solution $U = \vec{p} 
 e^{ik\vec{x} \cdot \vec{d}}$ with $\vec{p} = (1,0,0)^T$, 
 $\vec{d} = (0,0,1)^T$ and $k = 1$ . 
 \item \textbf{Example 4:} We consider 
 $U = (\vec{d}\times(\vec{p}\times\vec{d})) e^{ik\vec{x}
\cdot \vec{d}}$ with $\vec{p} = (1.01, 0, 1.05)^T$,\\ 
$\vec{d} = (0.57735, 0.57735, 0.57735)^T
$ normalized, and $k = 2$. 
\end{itemize}

\subsubsection{Empiric convergence of norms of basis functions}

We validate the predicted convergences rate of vectorial basis functions. 
In particular, we compute the $\Hdiv$-norm for RWG basis functions $\tilde{
\beta}_e$. We do this by inspecting the diagonals of the Galerkin matrices of 
the electric field integral operator $\OE_k$ with wavenumber $k = i$. Note that 
calculating this gives a measure of the square of the norms of the basis 
functions, and thus we compare them with two times the computed convergence 
estimate. In Table~\ref{tab:Bempp_RWGconvergence} we see that the observed 
convergence rates for all of the basis functions are close to the expected ones.  

\begin{table}[h!]
\begin{center}
\begin{tabular}{|c|c|}\hline
 \# elements &  $max(diag(\MEi))$\\
 \hline
 24   & 4.153e-01\\\hline
 96   & 2.373e-01\\\hline
 384  & 1.174e-01\\\hline
 1536 & 5.851e-02\\\hline
6144 & 2.923e-02\\\hlinewd{1pt}
observed conv. rate & \green{0.968}\\ \hline
expected conv. rate & 1 \\\hline
\end{tabular}
\caption{Maximal element on diagonal of $\ME$ with wavenumber $k=i$ to estimate 
convergence of RWG basis functions } 
\label{tab:Bempp_RWGconvergence}
\end{center}
\end{table}

\subsubsection{Checking convergence of Calder\'on residuals}

In this section we verify our estimates for Calder\'on residuals. We
remark that \texttt{Bempp} uses SNC and RWG functions, which we will denote 
by $\mathcal{E}^0(\GGG_h)\subset \Hcurl$ and $\mathcal{E}^0_{\times}(\GGG_h) 
\subset \Hdiv$, respectively. The \texttt{Bempp} results can be found in
Table~\ref{tab:Bempp_Maxwell}. From Remark~\ref{rem:scaling}, we predict to 
have $\|\vect{\rho}_{M,i}\|_{\infty} = \mathcal{O}(h^2)$ and $\|\vect{\rho}_{
M,i}\|_{2} = \mathcal{O}(h)$ for $i=1,2$. We see that for both Examples we 
obtain a convergence rate of approximately one order higher than we 
expected.

\begin{table}[h!]
\begin{center}
\begin{tabular}{|c|c|c|c|c|}\hline
\multicolumn{5}{|c|}{Example 4}\\ \hline
\# elements & $\|\vect{\rho}_{M,1}\|_{2}$ & $\|\vect{\rho}_{M,1}\|_{\infty}$& 
$\|\vect{\rho}_{M,2}\|_{2}$ & $\|\vect{\rho}_{M,2}\|_{\infty}$\\ \hline
 24   & 2.16776e-01 & 5.76130e-02& 3.62201e-01 & 8.23841e-02\\\hline
 96   & 4.64068e-02 & 8.96662e-03&6.41830e-02 & 1.15569e-02\\\hline
 384  & 1.01251e-02 & 1.11766e-03&1.19495e-02 & 1.31189e-03\\\hline
 1536 & 2.37529e-03 & 1.48827e-04& 2.54622e-03 & 1.62684e-04\\\hline
6144 &  5.74156e-04 & 1.91928e-05&  5.90082e-04 & 2.00978e-05\\\hlinewd{1pt}
observed conv. rate &  \green{2.14}& \green{2.90}&  \green{2.32}& \green{3.02}\\\hline
expected conv. rate &1&2 &1&2 \\ \hline
\end{tabular}\vspace{0.05in}
\begin{tabular}{|c|c|c|c|c|}\hline
\multicolumn{5}{|c|}{Example 5}\\ \hline
 \# elements & $\|\vect{\rho}_{M,1}\|_{2}$ & $\|\vect{\rho}_{M,1}\|_{\infty}$& 
 $\|\vect{\rho}_{M,2}\|_{2}$ & $\|\vect{\rho}_{M,2}\|_{\infty}$\\ \hline
 24   & 1.88446e-01 & 7.66625e-02 & 2.86437e-01 & 8.04880e-02\\\hline
 96   & 3.22867e-02 & 5.74773e-03& 4.73946e-02 & 9.75907e-03\\\hline
 384  & 5.69888e-03 & 4.92157e-04& 7.76580e-03 & 8.36725e-04\\\hline
 1536 & 1.19773e-03 & 7.17868e-05& 1.43842e-03 & 7.48340e-05\\\hline
6144 &  2.82986e-04 & 9.43459e-06&  3.08229e-04 & 8.56077e-06\\\hlinewd{1pt}
observed conv. rate &  \green{2.35}& \green{3.23} &  \green{2.48}& \green{3.22}\\\hline
 expected conv. rate &1&2 &1&2  \\ \hline
\end{tabular}
\end{center}
\caption{Calder\'on residuals for Maxwell BVP.}
\label{tab:Bempp_Maxwell}
\end{table}

\subsection{Numerical results when introducing artificial errors}
\label{ssec:ErrorsMax}

As before, for each test, we calculate the Calder\'on residuals, compare them 
with the standard \emph{MMS} convergence test. In other words, we do 
the following convergence test: (i) We build the \emph{correct} Galerkin 
matrices and store them as $\ME$ and $\MH$. We also build the mass matrix 
$\MM$ arising from $\Id$; (ii) we introduce an artificial error in the 
computation of the matrix of $\OE$, and store the \emph{defective} Galerkin 
matrix as $\ME^e$; (iii) We use GMRES with relative tolerance $10^{-10}$.
to numerically solve the following linear systems 
\begin{align}
\ME^e \vec{w} = (\frac{1}{2}\MM - \MH)\tilde{\Ph}\gamma_{\times} u
\end{align}
for the coefficient vectors $\vec{w}$; (iv) We measure the 
\emph{MMS error}
\begin{equation}\label{eq:conv_error_maxwell}
 \mathsf{e}_R:= \|\gamma_R u - w_h\|_{\Hdiv}\approx ((\mathcal{P}_h
 \gamma_R u - \vec{w})^{\top} \ME (\mathcal{P}_h\gamma_R u - \vec{w}))^{\frac{1}{2}};
\end{equation}
(v) Finally, we repeat steps (i)-(iv) for a sequence of dyadically refined 
meshes and see if this behaves like $\O(h^{\frac{3}{2}})$, where $h$ is the 
meshwidth as before. Note that in all tables where we measure $\mathsf{e}_R$, 
the columns labelled `ooc $\mathsf{e}_R$', contain the observed convergence 
rate calculated using the error of two consecutive meshsizes.

\subsubsection{Inaccurate quadrature}
We adapt the quadrature for Example 3. Standard parameters in Bempp are order 4 
for both the singular and the regular quadrature. These are adapted to 1 for 
the singular and 2 for the regular quadrature, respectively. In 
Table~\ref{tab:Bempp_Maxwell6a}, we see that adapting the quadrature does not 
do enough damage to destroy the convergence of the Calder\'on residuals, and 
also the convergence result from the MMS test, i.e., from calculating 
\eqref{eq:conv_error_maxwell}, gives us a higher rate than expected. 

\begin{table}[h!]
\begin{center}
\begin{tabular}{|c|c|c|c|c|c|c|}\hline
\# elements & $\|\vect{\rho}_{M,1}\|_{2}$ & $\|\vect{\rho}_{M,1}\|_{\infty}$& 
$\|\vect{\rho}_{M,2}\|_{2}$ &$\|\vect{\rho}_{M,2}\|_{\infty}$& $\mathsf{e}_R$,
tol=$10^{-10}$& ooc $\mathsf{e}_R$\\ \hline
 24   & 2.079e-01& 6.594e-02& 3.301e-01& 7.537e-02& 1.211e00&\\\hline
 96   & 3.709e-02 & 7.372e-03& 5.446e-02 & 9.329e-03& 4.006e-01& 1.60\\\hline
 384  & 8.142e-03& 8.703e-04& 1.004e-02 & 1.045e-03& 1.117e-01&1.84\\\hline
 1536 & 1.897e-03 &1.081e-04& 2.082e-03 & 1.214e-04& 2.952e-02&1.92\\\hline
6144 & 4.574e-04  & 1.400e-05& 4.747e-04   & 1.490e-05 & 7.786e-03&1.92\\\hlinewd{1pt}
ooc & \green{2.19} & \green{3.05} & \green{2.36} & \green{3.09} 
&\green{1.83}& \\ \hline
eoc &1&2 &1&2 & 1.5&1.5\\ \hline
\end{tabular}
\caption{
Calder\'on residuals with \texttt{Bempp} 
for Maxwell BVP, considering Example 3 and using less quadrature points.} 
\label{tab:Bempp_Maxwell6a}
\end{center}
\end{table}

\subsubsection{Adapting entries in $\ME$}
In order to be consistent with the previous numerical experiments, we investigate 
what happens when we introduce the following artificial errors:
\begin{itemize}
\item \textsf{[Error A]} replacing the diagonal of $\ME$ with random numbers 
between 0 and 10,
\item \textsf{[Error B]} scaling the diagonal of $\ME$ with a factor $\frac{1}{2}$,
\item \textsf{[Error C]} scaling the diagonal of $\ME$ with a factor $10^3$,
\item \textsf{[Error D]} scaling the first half of the diagonal of $\ME$ with 
a factor $10^3$,
\item \textsf{[Error E]} scaling the diagonal of $\ME$ with a factor $h^{-1}$,
where $h$ is the meshwidth.
\end{itemize}

\begin{table}[h!]
\begin{center}
\begin{tabular}{|c|c|c|c|c|c|c|}\hline
\# elements & $\|\vect{\rho}_{M,1}\|_{2}$ & $\|\vect{\rho}_{M,1}\|_{\infty}$& 
$\|\vect{\rho}_{M,2}\|_{2}$ & $\|\vect{\rho}_{M,2}\|_{\infty}$ & $\mathsf{e}_R$, 
tol=$10^{-10}$ & ooc $\mathsf{e}_R$ \\ \hline
 24   & 6.713e-01& 2.218e-01&6.543e-01& 2.527e-01&4.998e-01&\\\hline
 96   & 6.565e-01 &1.612e-01&6.215e-01 &1.733e-01&1.525e-01&1.71\\\hline
 384  & 7.217e-01& 8.728e-02& 7.257e-01 &8.791e-02&4.158e-02&1.87\\\hline
 1536 & 7.248e-01 &4.410e-02& 7.063e-01&4.421e-02&1.209e-02&1.78\\\hline
6144 & 7.117e-01  & 2.209e-02& 7.064e-01 &2.209e-02&3.618e-03&1.74 \\\hlinewd{1pt}
ooc & \red{-0.03} & \red{0.85} & \red{-0.04} & \red{0.90}& \green{1.8} &\\ \hline
 eoc &1&2 &1&2 & 1.5&1.5\\ \hline
\end{tabular}
\caption{
Convergence results with \texttt{Bempp} for one instance of the Maxwell BVP 
using Example 3 and introducing \textsf{Error A} (replacing diagonal entries of 
$\MV$ with random numbers). }
\label{tab:ErrorA_E}
\end{center}
\end{table}

\textbf{Error A:} Table \ref{tab:ErrorA_E} shows the results when replacing the 
diagonal of the Galerkin matrix $\ME$ of $\ME$ with random numbers. We see that 
while the observed convergence rate for the Calder\'on residuals is less than 
we expected, the rate from the MMS test in \eqref{eq:conv_error_maxwell} 
is higher than expected. We therefore see that the Calder\'on residual test 
is able to detect the introduced error, while the standard MMSs test 
is not.

\textbf{Error B-D:} Tables~\ref{tab:ErrorB_E}-\ref{tab:ErrorD_E} report the 
results when we multiply the diagonal of $\ME$ by different $h$-independent 
constants. There, the observed convergence rates are everywhere larger than the 
expected rates, which means that both tests agree on their results.

\begin{table}[h!]
\begin{center}
\begin{tabular}{|c|c|c|c|c|c|c|}\hline
\# elements & $\|\vect{\rho}_{M,1}\|_{2}$ & $\|\vect{\rho}_{M,1}\|_{\infty}$& 
$\|\vect{\rho}_{M,2}\|_{2}$ & $\|\vect{\rho}_{M,2}\|_{\infty}$& $\mathsf{e}_R$, 
tol=$10^{-10}$& ooc $\mathsf{e}_R$\\ \hline
  24 & 1.185e-01& 3.399e-02& 2.717e-01& 5.778e-02&7.151e-01& \\\hline
  96 & 2.805e-02 & 4.590e-03& 3.777e-02 & 6.276e-03&1.120e-01 &2.67\\\hline
 384 & 8.793e-03& 8.437e-04& 8.174e-03 & 8.027e-04&3.546e-02 &1.66\\\hline
1536 & 2.608e-03 & 1.205e-04& 2.4311e-03 &1.176e-04&1.236e-02 &1.52\\\hline
6144 & 7.159e-04  & 1.558e-05& 6.889e-04   & 1.539e-05&4.909e-03&1.33\\\hlinewd{1pt}
ooc & \green{1.81} & \green{2.74} & \green{2.12} & \green{2.95}  
&\green{1.8}&\\ \hline
 eoc &1&2 &1&2 &1.5&1.5\\ \hline
\end{tabular}
\caption{Convergence results with \texttt{Bempp} for Maxwell BVP using Example 
3 and introducing \textsf{Error B} (diving values on diagonal of $\ME$ by two).}
\label{tab:ErrorB_E}
\end{center}
\end{table}

\begin{table}[h!]
\begin{center}
\begin{tabular}{|c|c|c|c|c|c|c|}\hline
 \# elements & $\|\vect{\rho}_{M,1}\|_{2}$ & $\|\vect{\rho}_{M,1}\|_{\infty}$& 
 $\|\vect{\rho}_{M,2}\|_{2}$ & $\|\vect{\rho}_{M,2}\|_{\infty}$&  $\mathsf{e}_R$, 
 tol=$10^{-10}$& ooc $\mathsf{e}_R$\\\hline
 24   & 2.783e02 & 6.276e01& 2.784e02 & 6.278e01& 4.8284e-01&\\\hline
 96   & 9.028e01 & 1.916e01&9.030e01 & 1.916e01& 1.4313e-017&1.75\\\hline
 384  & 2.469e01& 2.424e00&2.469e01 & 2.424e00&4.1327e-02&1.79\\\hline
 1536 & 6.401e00 &3.040e-01& 6.401e00 & 3.040e-01& 1.2059e-02&1.78\\\hline
6144 & 1.626e00  & 3.803e-02&1.626e00   & 3.803e-02& 3.6064e-03&1.74\\\hlinewd{1pt}
ooc & \green{1.87} & \green{2.74} & \green{1.87} & \green{2.74}  
&\green{1.8}&\\ \hline
eoc &1&2 &1&2&1.5 &1.5\\ \hline
\end{tabular}
\caption{Convergence results with \texttt{Bempp} for Maxwell BVP using Example 
3 and introducing \textsf{Error C} (scaling diagonal of $\ME$ by $10^3$).}
\label{tab:ErrorC_E}
\end{center}
\end{table}

\begin{table}[h!]
\begin{center}
\begin{tabular}{|c|c|c|c|c|c|c|}\hline
 \# elements & $\|\vect{\rho}_{M,1}\|_{2}$ & $\|\vect{\rho}_{M,1}\|_{\infty}$& 
 $\|\vect{\rho}_{M,2}\|_{2}$ & $\|\vect{\rho}_{M,2}\|_{\infty}$& $\mathsf{e}_R$, 
 tol=$10^{-10}$& ooc $\mathsf{e}_R$\\\hline
 24   & 1.911e02 & 6.276e01& 2.012e02 & 6.278e01& 8.5126e-01&\\\hline
 96   & 6.126e01 & 1.916e01& 6.338e01 & 1.916e01&1.3188e-01&2.69\\\hline
 384  & 1.692e01& 2.424e00& 1.732e01 & 2.424e00&4.6648e-02&1.78\\\hline
 1536 & 4.465e00 &3.040e-01& 4.498e00 & 3.040e-01& 1.4739e-02&3.16\\\hline
6144 & 1.142e00  & 3.803e-02& 1.147e00   & 3.803e-02& 4.2723e-03 &1.79\\\hlinewd{1pt}
ooc & \green{1.86} & \green{2.74} & \green{1.87} & \green{2.74}& 
\green{1.8} &\\ \hline
 eoc &1&2 &1&2&1.5 &1.5\\ \hline
\end{tabular}
\caption{Convergence results with \texttt{Bempp} for Maxwell BVP using Example 
3 and introducing \textsf{Error D} (scaling first half of diagonal of $\ME$ by 
$10^3$).}
\label{tab:ErrorD_E}
\end{center}
\end{table}

\pagebreak
\textbf{Error E:} Table~\ref{tab:ErrorE_E} displays the results when 
scaling the diagonal of $\ME$ by a factor $h^-1$, where $h$ is the meshwidth 
(computed as average size of an element in \texttt{Bempp}). We see that while 
the observed convergence rates for Calder\'on residuals are lower than expected, 
the rate in $\Hdiv$-norm still is higher than we expected. Therefore, the 
Calder\'on residual test detects the introduced error, while the standard 
test does not.

\begin{table}[h!]
\begin{center}
\begin{tabular}{|c|c|c|c|c|c|c|}\hline
 \# elements & $\|\vect{\rho}_{M,1}\|_{2}$ & $\|\vect{\rho}_{M,1}\|_{\infty}$& 
 $\|\vect{\rho}_{M,2}\|_{2}$ & $\|\vect{\rho}_{M,2}\|_{\infty}$& $\mathsf{e}_R$, 
 tol=$10^{-10}$& ooc $\mathsf{e}_R$\\\hline
 24   & 1.278e00 &2.669e-01& 1.381e00 & 2.916e-01&1.1026e00&\\\hline
 96   & 8.329e-01 &1.770e-01 & 8.480e-01 & 1.799e-01& 1.4364e-01&2.94\\\hline
 384  &4.667e-01 & 4.620e-02 & 4.688e-01 & 4.641e-02&4.1157e-02&1.80 \\\hline
 1536 & 2.461e-01 & 1.176e-02& 2.463e-01 & 1.178e-02&1.2049e-02&1.77\\\hline
6144 &  1.263e-01 & 2.963e-03 & 1.263e-01 & 2.964e-03&3.6062e-03&1.74 \\\hlinewd{1pt}
ooc & \red{0.84} & \red{1.69} & \red{0.87} & \red{1.72}  & 
\green{2.0}&\\ \hline
eoc &1&2 &1&2&1.5&1.5 \\ \hline
\end{tabular}
\caption{Convergence results with \texttt{Bempp} for Maxwell BVP using Example 
3 and introducing \textsf{Error E} (scaling diagonal of $\ME$ by $h^{-1}$).}
\label{tab:ErrorE_E}
\end{center}
\end{table}

\section{Summary and Discussion}

In this summary, we use the following notation:
\begin{itemize}
 \item `\xmark' if the convergence rate is lower than expected, which means the test 
says ''example fails to be correct``;
\item `\cmark' if the convergence rate is higher 
or deviates less than $10\%$ from what we predicted, which means the test says ''example is correct``;
\item `*' when the obtained errors for all refinement meshes are within machine 
precision, so we cannot use the convergence rate to make a conclusion, but it is clear that no errors were detected;
\item `$\circledast$' when the convergence rates for the `MMS test' fluctuate, from which we conclude that the test fails.
\end{itemize}
The idea is to compare the outputs of the Calder\'on residual test (dubbed 
`Calder\'on test' for short), and the MMS tests from \eqref{eq:conv_test_1},
\eqref{eq:conv_test_2}, and \eqref{eq:conv_error_maxwell} (labelled `MMS' in the tables). 
In 
particular, we do not think it is a problem that the Calder\'on residual test 
is unable to detect errors in the matrix entries when the standard MMS 
test also fails to notice. Hence, we also consider in our summary whether the 
outputs agree or not. 
We remark that for the MMS test, our conclusion follows the convergence rates calculated between 
consequtive refinement levels and not the linear regression based on all results, 
since in some cases there are large differences between the computation methods.  

\begin{table}[h!]
\begin{center}
\begin{tabular}{|c||c|c|c||c|c|c|}\hline
\multicolumn{7}{|c|}{Weakly singular operator $\OV$}\\ \hline
& \multicolumn{3}{c||}{Example 1a}& \multicolumn{3}{c|}{Example 2}\\ \hline
Artificial Error & Calder\'on test & MMS & Agree? & Calder\'on test 
& Norm test & Agree?\\ \hline
\textsf{Error A} & \xmark  & $\circledast$ & yes & \xmark& \xmark & yes \\ \hline
\textsf{Error B} & \cmark  & \cmark & yes & \xmark& \xmark & yes \\ \hline
\textsf{Error C} & \cmark  & * & yes & \cellcolor[gray]{0.9}\cmark& \cellcolor[gray]{0.9}\xmark & \cellcolor[gray]{0.9}no \\ \hline
\textsf{Error D} & \cmark  & * & yes & \cellcolor[gray]{0.9}\cmark& \cellcolor[gray]{0.9}\xmark & \cellcolor[gray]{0.9}no \\ \hline
\textsf{Error E} & \cellcolor[gray]{0.9}\xmark  & \cellcolor[gray]{0.9}* &\cellcolor[gray]{0.9}no & \xmark& \xmark & yes \\ \hline
\end{tabular}\hspace{0.5in}
\begin{tabular}{|c||c|c|c||c|c|c|}\hline
\multicolumn{7}{|c|}{Hypersingular operator $\OW$}\\ \hline
& \multicolumn{3}{c||}{Example 1a}& \multicolumn{3}{c|}{Example 2}\\ \hline
Artificial Error & Calder\'on test & MMS & Agree? & Calder\'on test 
& Norm test & Agree?\\ \hline
\textsf{Error A} & \xmark  & $\circledast$ & yes & \xmark  & \xmark & yes \\ \hline
\textsf{Error B} & \xmark  & \xmark & yes  & \xmark  & \xmark & yes\\ \hline
\textsf{Error C} & \cellcolor[gray]{0.9}\xmark  & \cellcolor[gray]{0.9}* & \cellcolor[gray]{0.9}no & \xmark  & \xmark & yes \\ \hline
\textsf{Error D} & \xmark  & \xmark & yes & \xmark  & \xmark & yes \\ \hline
\textsf{Error E} &\cellcolor[gray]{0.9} \xmark  & \cellcolor[gray]{0.9}* & \cellcolor[gray]{0.9}no & \xmark  & \xmark & yes \\ \hline
\end{tabular}
\end{center}
\caption{Summary of convergence results for Laplace BPV.}
\label{tab:sumLap}
\end{table}

It turns out that the Calder\'on residual estimates that we proposed in this 
paper did not provide the robust debugging tool that we had hoped for. Yet, 
from our numerical experiments, we believe that it can still be useful in 
some circumstances. Indeed, we see in Table~\ref{tab:sumLap} that in the case 
of $\OW$ for Example 2 of the Laplacian, the Calder\'on residual test is always able to 
detect errors. However, for that BVP, it is unreliable for $\OV$. 
Even if it does work for some introduced errors, this is not helpful because 
we do not know a priori what the problem is when debugging. 

\begin{table}[h!]
\begin{center}
\begin{tabular}{|c||c|c|c|}\hline
\multicolumn{4}{|c|}{Electric Field Integral operator $\OE_k$}\\ \hline
Artificial Error & Calder\'on test & Norm test & Agree? \\ \hline
\textsf{Error A} & \cellcolor[gray]{0.9}\xmark  &\cellcolor[gray]{0.9} \cmark & \cellcolor[gray]{0.9}no \\ \hline
\textsf{Error B} & \cmark  & \cmark & yes  \\ \hline
\textsf{Error C} & \cmark  & \cmark & yes \\ \hline
\textsf{Error D} & \cmark  & \cmark & yes \\ \hline
\textsf{Error E} & \cellcolor[gray]{0.9}\xmark  & \cellcolor[gray]{0.9}\cmark &\cellcolor[gray]{0.9} no \\ \hline
\end{tabular}
\end{center}
\caption{Summary of convergence results for Maxwell BPV.}
\label{tab:sumMaxwell}
\end{table}

In the case of the EFIE, the MMS test seems to always be oblivious to 
the introduced errors. That means that these do not affect the solution of the 
boundary integral equation, and hence, it is hard to decide whether it is 
relevant that the Calder\'on residual test does ring the alarm in some cases. 

In spite of our disappointment, we believe it is worth reporting these findings.
On the one hand, it may be useful for some practitioners, and then it is 
important that they are aware of the shortcomings of this technique. On the other 
hand, it may be possible in the future to make these estimates sharper.

\section*{Acknowledgments}
The authors based in Delft would like to thank Kees Vuik for his valuable feedback 
in an earlier version of the numerical results Section. A.W. would like to acknowledge D. Hoonhout 
and J. Wallaart for their IT support.  
This work has been partially funded by the Dutch Research Council (NWO) 
under the NWO-Talent Programme Veni with the project number VI.Veni.212.253.

\bibliographystyle{plain}
\bibliography{bibliography}

\appendix
\section{Proof of Lemma~\ref{lemma:equiv_hcurl_hdiv}}
\label{app:Lemma}
\begin{proof}
Let us recall the following definitions and relations between spaces from
\cite{BuC01-I, BuC01-II,BCS02,BuH03}.
\begin{itemize}
\item $H^{1/2}_T(\Gamma):=\{\gamma_T(u):= n\times(u|_{\Gamma}\times n), 
u\in H^1(\Omega)^3\}$, \\$\|p\|_{H^{1/2}_T(\Gamma)}:= \inf \{\|u\|_{H^1(
\Omega)}|\gamma_T u = p\}$,
\item $H^{1/2}_R(\Gamma):= \{v| n\times v \in H^{1/2}_T(\Gamma)\}$, 
$\|q\|_{H^{1/2}_{R}(\Gamma)}:= \|n\times q\|_{H^{1/2}_T(\Gamma)}$,
\item $H^{-1/2}_T(\Gamma):= H^{1/2}_R(\Gamma)'$, 
$H^{-1/2}_R(\Gamma):= H^{1/2}_T(\Gamma)'$ (using the symmetric 
pairing)
\item $\Hdiv := \{\mu \in H^{-1/2}_R(\Gamma)| \div_{\Gamma}\mu \in 
H^{-1/2}(\Gamma)\}$,\\ $\|\mu\|_{\Hdiv}:= \|\mu\|_{H^{-1/2}_R(
\Gamma)}+ \|\div_{\Gamma}\mu\|_{H^{-1/2}(\Gamma)}$
\item $\Hcurl:= \{\mu\in H^{-1/2}_T(\Gamma)| \bcurl_{\Gamma}\mu \in 
H^{-1/2}(\Gamma)\}$, \\
$\|\mu\|_{\Hcurl}:= \|\mu\|_{H^{-1/2}_T(\Gamma)}+ \|\bcurl_{\Gamma}\mu
\|_{H^{-1/2}(\Gamma)}$.
\end{itemize}
Since it holds that $\bcurl_{\Gamma} v = \div_{\Gamma}(\Vn\times v)$ for $v\in 
\Hcurl$, we only need to show that 
\begin{equation}\label{eq:equivalence_R,T}
  \|v\|_{H^{-1/2}_T(\Gamma)} =\|\Vn\times v\|_{H^{-1/2}_R(\Gamma)}.
\end{equation}
Let $\OR u := \Vn\times u : H^{1/2}_R(\Gamma)\to H^{1/2}_T(\Gamma)$. 
From \cite[Equation (16) and below]{BCS02} it then follows that $\OR^{-1} = 
\OR^{\star} = -\OR$. Hence
\begin{align}
\|v\|_{H^{-1/2}_T(\Gamma)} &= 
\sup_{w \in H^{1/2}_R \setminus \lbrace0\rbrace} \frac{|\langle v,w
\rangle_{H^{-1/2}_T(\Gamma)\times H^{1/2}_R(\Gamma)}|}
{\|w\|_{H^{1/2}_R(\Gamma)}}\nonumber\\
&= \sup_{\OR w \in H^{1/2}_R \setminus \lbrace0\rbrace} 
\frac{|\langle \OR^{\star} \OR v,w\rangle_{H^{-1/2}_T(\Gamma)\times H^{1/2}_R(
\Gamma)}|}{\|\OR w\|_{H^{1/2}_T(\Gamma)}}\nonumber\\
&= \sup_{u \in H^{1/2}_R \setminus \lbrace0\rbrace} \frac{|\langle \OR 
v,u\rangle_{H^{-1/2}_R(\Gamma)\times H^{1/2}_T(\Gamma)}|}{
\|u\|_{H^{1/2}_T(\Gamma)}} = 
\|\OR v\|_{H^{-1/2}_R(\Gamma)}.
\end{align}

\end{proof}

\section{Additional observations}
Consider the following setting: let $\Omega$ be the unit cube in 3D, and choose real numbers $a_1, a_2$ and $a_3$ such that $\sum_{i=1}^{3}a_i = 0$. Then
\begin{equation*}
  u(x,y,z) = a_1 x^2 + a_2 y^2 + a_3 z^2
\end{equation*}
is a solution to the Laplace equation on $\Omega$. Then 

\begin{equation*}
  \tau_N u  = \begin{pmatrix}
                2 a_1 x\\ 2 a_2 y \\ 2a_3 z
              \end{pmatrix} \cdot \vect{n} |_{\Gamma} = \begin{cases}
                                                          0,\quad x= 0\\
                                                          2a_1,\quad x=1\\
                                                          0,\quad y= 0\\
                                                          2a_2,\quad y=1\\
                                                          0,\quad z= 0\\
                                                          2a_3,\quad z=1,
                                                        \end{cases}
\end{equation*}
which consists of piecewise constant functions. This means that $\OI_h^{p-1,-1} \tau_N u = \tau_N u$ in this case. In Example 2 we are in this situation. Then, the residuals \eqref{eq:residual_D} and \eqref{eq:residual_N} become 
 \begin{align}
r_D(\psi) :=& \int_{\Gamma}[(\frac{1}{2} \Id -\OK)\OI_h^{p,0} (\tau_D u)
](x) \psi(x) \d S(x)+ \OV(\OI_h^{p-1,-1}\tau_N u)](x) \psi(x) \d S(x) \nonumber\\
=& \int_{\Gamma}[(\frac{1}{2} \Id -\OK)(\OI_h^{p,0} -\Id)(\tau_D u)
](x) \psi(x) \d S(x), 
\end{align}
and
\begin{align}
r_N(\phi) :=& \int_{\Gamma}[\OW \OI_h^{p,0}(\tau_D u)
](x) \phi(x) \d S(x)+ (\frac{1}{2}\Id +\OK')\OI_h^{p-1,-1}\tau_N u](x) \phi(x) \d S(x)\nonumber\\
=& \int_{\Gamma}[\OW (\OI_h^{p,0}-\Id)(\tau_D u)
](x) \phi(x) \d S(x), 
\end{align}
which are slightly simpler forms, and might help with identifying which terms contribute to the fact that we numerically observe higher convergence rates. 

\section{Eigenvalues corresponding to matrices with artificial errors for $\MV$}\label{app:eig}

In this Appendix we take a closer look at what is happening when we introduce 
\textsf{Error B}, i.e., when we scale the diagonal of $\MV$ with a factor $1/2$.
Figures~\ref{fig:LaplaceBempp_ev1a} and \ref{fig:LaplaceBempp_ev} display the 
eigenvalues of the original matrix $\MV$ and the resulting scaled matrix $\MV^e$ 
for Examples 1a and 2, respectively. In both cases, we are considering the 
smallest mesh. We see in these plots that scaling the 
diagonal of $\MV$ with a factor $\frac{1}{2}$ shifts down the spectra. When 
this is done, quite a lot of the eigenvalues of $\MV^e$ are now negative, and 
the matrix is thus no longer symmetric positive definite. This may be one 
of the reasons why the iterative solvers we tried, namely CG and GMRES, failed. 

\begin{figure}[h!]
\centering
\begin{subfigure}[t]{0.49\textwidth}
	\includegraphics[width=\textwidth, trim={0 0.5cm 0 1cm},clip]{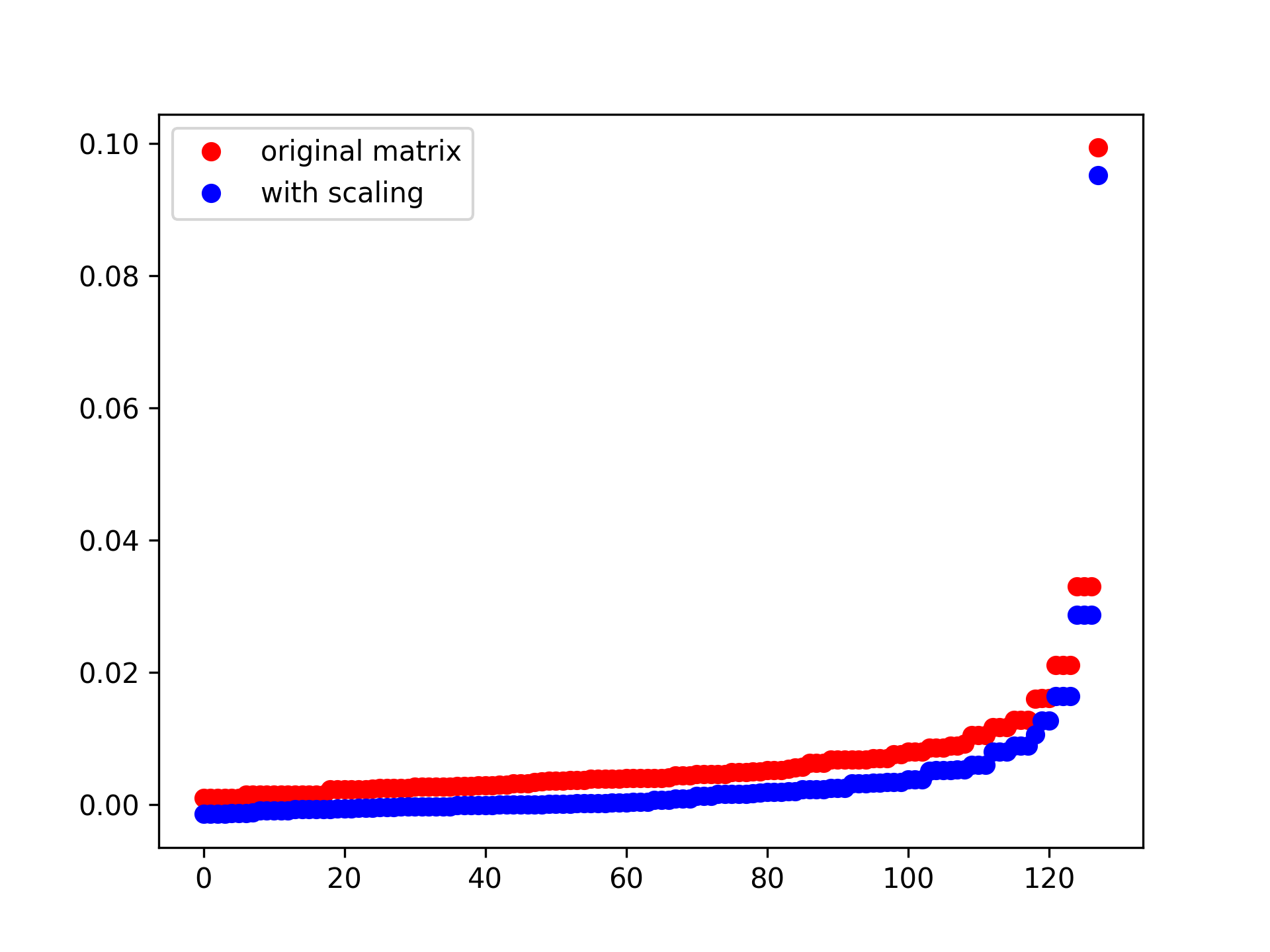}
	\caption{All eigenvalues.}
	\label{fig:ev}
\end{subfigure} 
\hfill
\begin{subfigure}[t]{0.49\textwidth}
	\includegraphics[width=\textwidth, trim={0 0.5cm 0 1cm},clip]{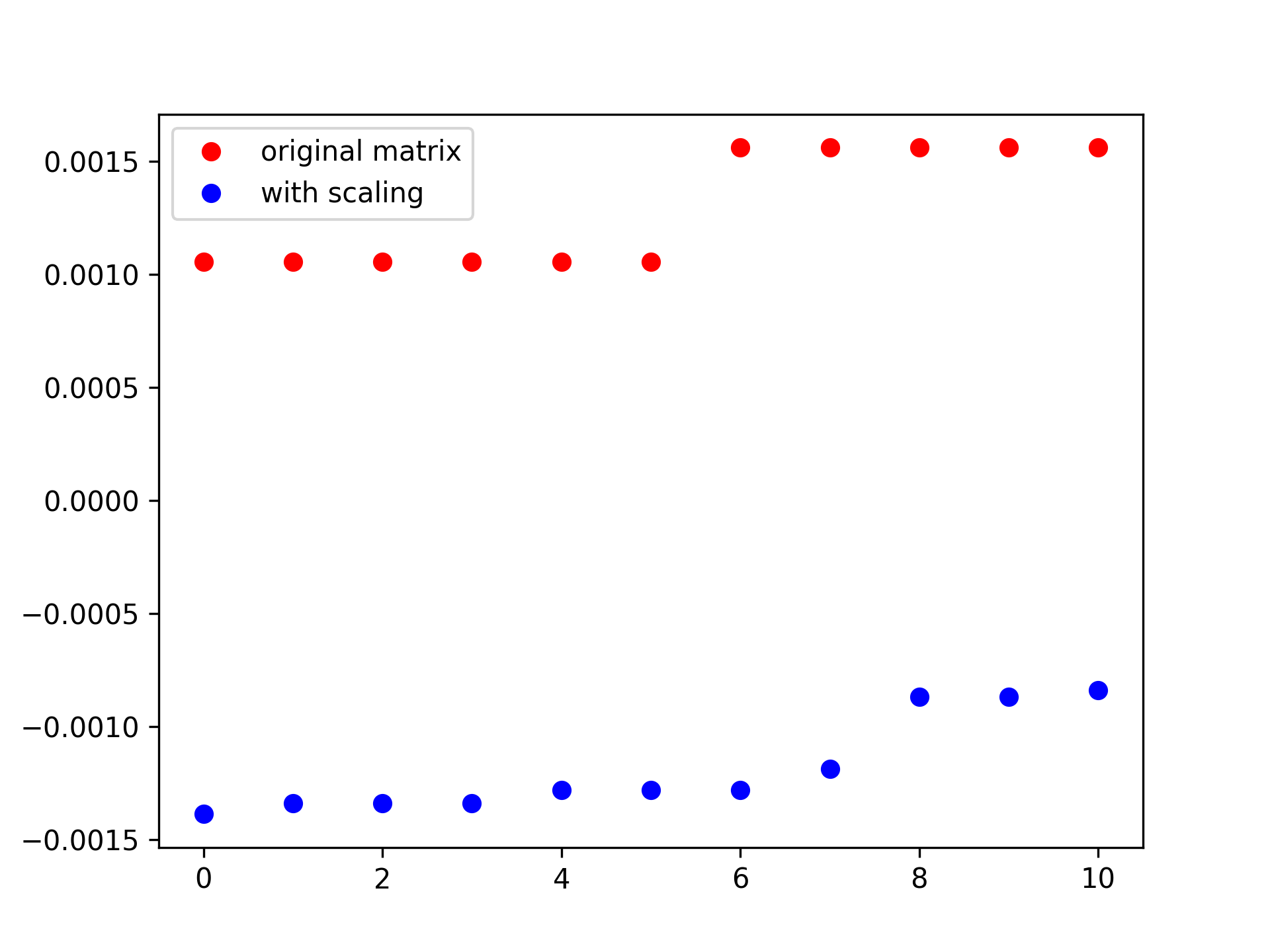}
	\caption{The first ten eigenvalues.}
	\label{fig:first_10_ev}
\end{subfigure} 
\vspace{-0.1in}
\caption{Eigenvalues when scaling the diagonal of $\MV$ by $1/2$ in Example 1a 
for the Laplacian (\textsf{Error B}). Results for smallest mesh, i.e., $N= 128$ elements.}
\label{fig:LaplaceBempp_ev1a}
\end{figure}

\begin{figure}[h!]
\centering
\begin{subfigure}[t]{0.49\textwidth}
	\includegraphics[width=\textwidth, trim={0 0.5cm 0 1cm},clip]{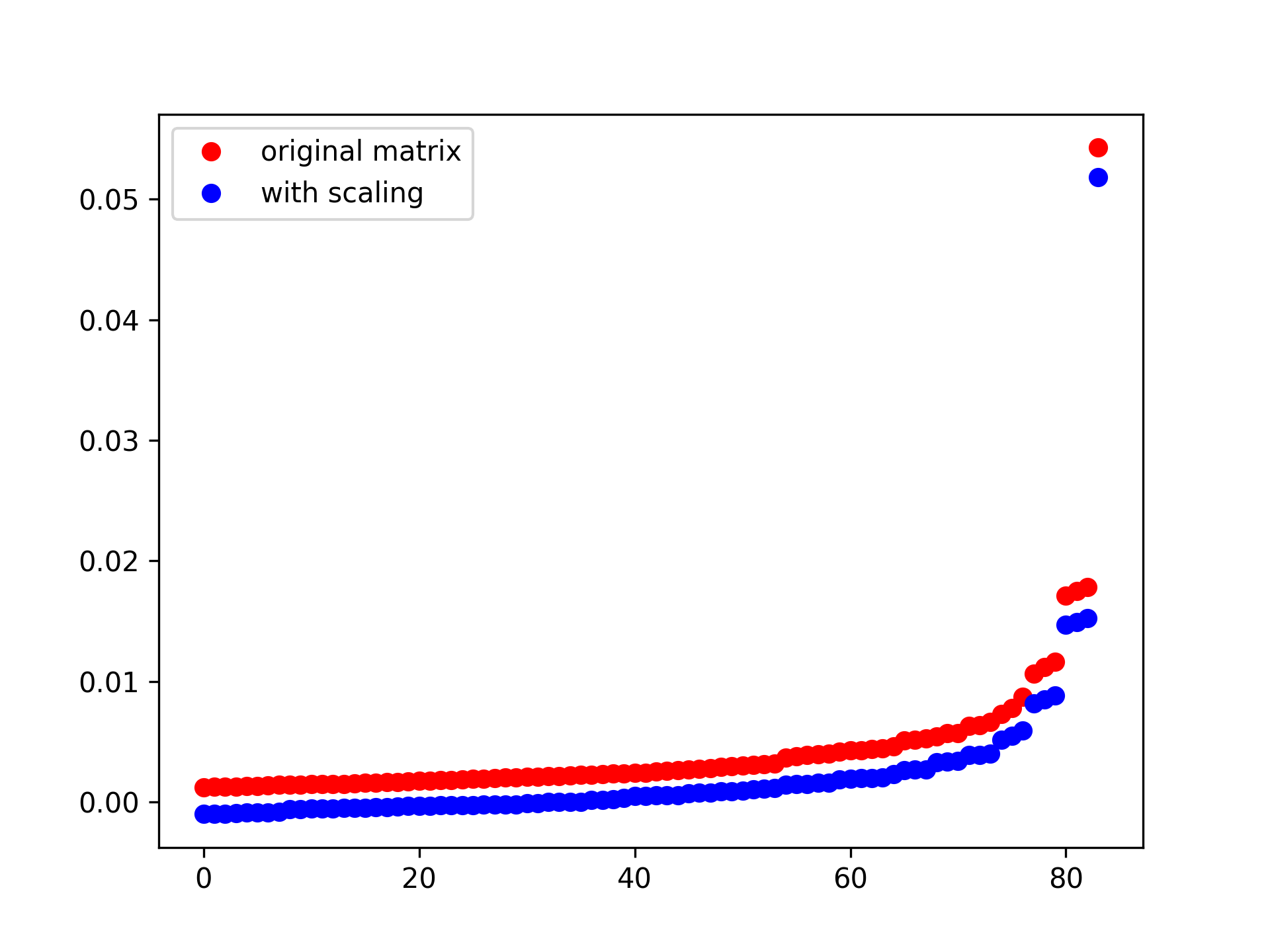}
	\caption{All eigenvalues.}
	\label{fig:ev}
\end{subfigure} 
\hfill
\begin{subfigure}[t]{0.49\textwidth}
	\includegraphics[width=\textwidth, trim={0 0.5cm 0 1cm},clip]{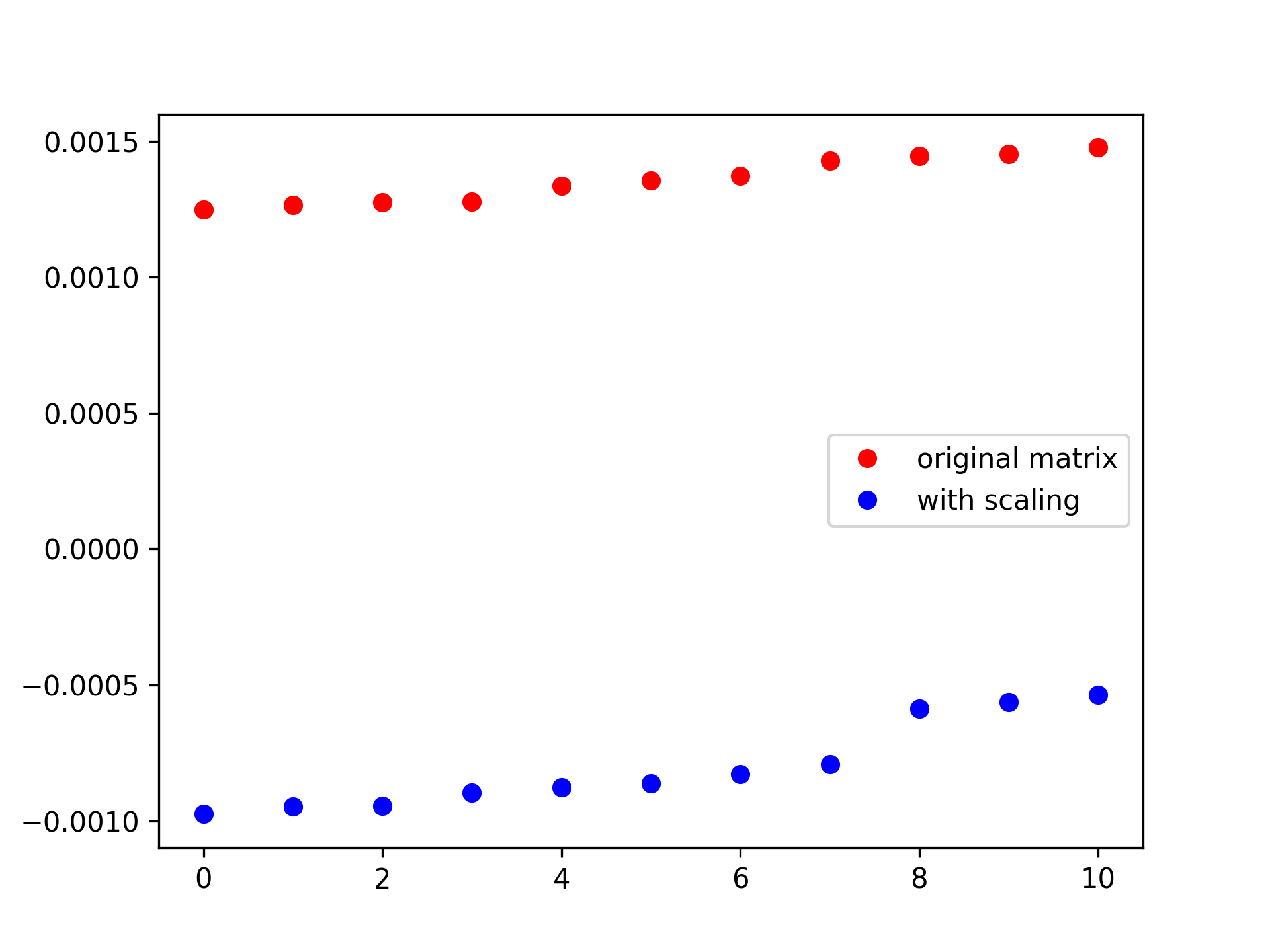}
	\caption{The first ten eigenvalues.}
	\label{fig:first_10_ev}
\end{subfigure} 
\vspace{-0.1in}
\caption{Eigenvalues when scaling the diagonal of $\MV$ by $1/2$ in Example 2 
for the Laplacian (\textsf{Error B}). Results for smallest mesh, i.e., $N= 84$ elements.}
\label{fig:LaplaceBempp_ev}
\end{figure}

\end{document}